\documentclass[11pt]{article}
\usepackage{parskip}
\usepackage{amsmath,amsfonts,amssymb,amsthm,mathtools,bm}
\usepackage{bbold}
\usepackage{enumitem}   
\usepackage[T1]{fontenc}
\usepackage[margin=2.54cm]{geometry}
\usepackage{setspace}
\usepackage{graphicx}
\usepackage{standalone}
\usepackage{float}
\usepackage{thmtools}
\usepackage{tikz}
\usepackage{tikzstyle}
\usepackage{stmaryrd}
\usepackage{subcaption}
\usepackage{url}
\usepackage[linktoc=all,hidelinks,colorlinks,unicode=true]{hyperref} 
\usepackage[capitalise,compress,nameinlink,noabbrev]{cleveref} 
\hypersetup{linkcolor={blue!70!black},citecolor={blue!70!black},urlcolor={blue!70!black}}
\usepackage{authblk}

\DeclareMathOperator{\E}{\mathbb{E}}
\DeclareMathOperator{\Prob}{\mathbb{P}}

\newcommand{\R}{\mathbb{R}}

\newcommand{\norm}[1]{\left\lVert#1\right\rVert}
\newcommand{\defn}[1]{\emph{#1}} 
\newcommand{\bigO}[0]{O} 
\DeclarePairedDelimiter\abs{\lvert}{\rvert}

\newcommand{\vecf}{h}

\newtheorem{theorem}{Theorem}[section]
\newtheorem{claim}[theorem]{Claim}
\newtheorem{lemma}[theorem]{Lemma}
\newtheorem{conjecture}[theorem]{Conjecture}
\newtheorem{problem}[theorem]{Problem}

\newtheorem{corollary}[theorem]{Corollary}

\newenvironment{poc}{\begin{proof}}{\end{proof}}

\usepackage[backend=biber, style=numeric, firstinits=true, url=false,eprint=false,maxbibnames=99,doi=true]{biblatex}
\title{Colour-balanced subgraphs}
\author{
Emma Hogan\footnote{Mathematical Institute, University of Oxford, Oxford OX2 6GG, UK,\\ \texttt{\{emma.hogan, alexander.scott, dmitry.tsarev\}@maths.ox.ac.uk}} \footnote{Supported by EPSRC grant EP/W524311/1.}\qquad
Alex Scott\protect\footnotemark[1] \footnote{Supported by EPSRC grant EP/X013642/1.}\qquad
Dmitry Tsarev\protect\footnotemark[1]
\footnote{Supported by EPSRC grant EP/W523781/1.}}
\date{\today}

\begin{document}

\maketitle

\begin{abstract}
    A $k$-edge-coloured graph is \defn{colour-balanced} if each colour appears equally often. Resolving a conjecture of Pardey and Rautenbach, we show that any colour-balanced $k$-edge-coloured complete graph $K_{2kt}$ contains a perfect matching that can be made colour-balanced by recolouring $O(k^2)$ edges. More generally, we obtain analogous bounds for arbitrary bounded-degree spanning subgraphs of edge-coloured complete graphs and for perfect matchings in edge-coloured $r$-uniform complete hypergraphs in a more general vector-label setting. The former result answers a question recently posed by Banerjee and Hollom, and significantly improves earlier bounds for all previously studied classes of subgraph. Our proofs reduce each of these problems to a setting in which we can apply a bound for perfect matchings in the complete bipartite graph, established via a linear relaxation and a necklace-splitting argument.
\end{abstract}
 
\section{Introduction}

A recurring theme in extremal combinatorics is the problem of identifying substructures that are, in some sense, representative of a larger combinatorial object. In graph theory, one formulation of this problem is as follows: given an edge-labelled graph, can one find a subgraph of a prescribed type whose label distribution is close to that of the host graph? Questions of this type have been studied in the context of $k$-edge-coloured complete graphs for a range of substructures, including perfect matchings (\cite{caro2022zero,Ehard2020,hollom2025uniform,kittipassorn2020,pardey2022}), spanning forests (\cite{caro2022zero,hollom2024,mohr2022zero,pardey2023efficiently}), and factors (\cite{banerjee2026}), with various bounds obtained on how closely such a subgraph can reflect the global distribution of colours; however, many natural problems remain open. The aim of this paper is to improve known bounds on a range of such problems. 

Problems on representative subgraphs have their roots in zero-sum Ramsey theory. The underlying question in this field is to find a copy of a fixed subgraph $H$ in a labelled clique $K_n$ such that the sum of the edge labels over the copy of $H$ is zero. When $K_n$ is labelled by colours $\{-1,1\}$, a zero-sum subgraph corresponds to a subgraph with an equal number of edges of each colour, and is therefore representative of a host graph with an equal number of edges of each colour. Variants of the problem where $E(K_n)$ is labelled by some finite abelian group have also received considerable attention. For a survey of early developments on zero-sum graph problems, see~\cite{caro1996}, and for early work on almost zero-sum spanning subgraph problems see~\cite{caro2016zero}.

Inspired by early work of Bialostocki and Dierker~\cite{bialostocki1992} on zero-sum matchings, Caro, Hansberg, Lauri and Zarb~\cite{caro2022zero} asked whether it was always possible to find a zero-sum perfect matching in a zero-sum complete graph $K_{4n}$ with edges labelled by $\{-1,1\}$. The question was answered affirmatively by Ehard, Mohr and Rautenbach~\cite{Ehard2020} and independently by Kittipassorn and Sinsap~\cite{kittipassorn2020}. Kittipassorn and Sinsap further initiated the study of so-called colour-balanced subgraphs with more than $2$ colours.

Let $G$ be a graph and let $c \colon E(G) \to [k]$ be a $k$-edge-colouring of $G$. We call both $G$ and $c$ \defn{colour-balanced} if $\abs{c^{-1}(i)} = \abs{c^{-1}(j)}$ for each $i, j \in [k]$. Motivated by the work of Caro, Hansberg, Lauri and Zarb, Kittipassorn and Sinsap~\cite{kittipassorn2020} asked whether, for any positive integer $t$, every colour-balanced $k$-edge-colouring of the complete graph $K_{2kt}$ admits a colour-balanced perfect matching $M$. This question was answered negatively by Pardey and Rautenbach \cite{pardey2022}, who found an example of a $3$-edge-coloured $K_6$ admitting no colour-balanced perfect matching. Instead, they asked for a bound on the colour imbalance of an optimal perfect matching. Define 
\[
    f_c(M) = \sum_{i=1}^k \abs{\abs{c^{-1}(i) \cap M} - t}
\]
to be the total deviation from being colour-balanced across all colour classes, noting that in a colour-balanced perfect matching of $K_{2kt}$, every colour appears exactly $t$ times. Pardey and Rautenbach proved that, when $k=3$, there exists a perfect matching $M$ of $K_{6t}$ satisfying $f_c(M) \leq 2$, and for general $k$ they conjectured the following. 
\begin{conjecture}[Conjecture 1, \cite{pardey2022}]
    For all integers $k \geq 2$ and $t \geq 1$, every colour-balanced $k$-edge-coloured $K_{2kt}$ admits a perfect matching $M$ satisfying $f_c(M) = \bigO(k^2)$.
\end{conjecture}
Making progress towards this conjecture, they found that every colour-balanced $k$-edge-colouring of $K_{2kt}$ admits a perfect matching $M$ satisfying $f_c(M) \leq 3k\sqrt{kt\log 2k}$. Further recent progress by Hollom~\cite{hollom2025uniform} removed the dependence on $t$ from this bound, establishing the existence of a perfect matching $M$ satisfying $f_c(M) \leq 4^{k^2}$. In the following result, we let $n=kt$ for simplicity. Our first result improves on the bound of Hollom and resolves the conjecture of Pardey and Rautenbach.
\begin{theorem} \label{thm:complete_colours}
    Let $c$ be a colour-balanced $k$-edge-colouring of $K_{2n}$. Then there exists a perfect matching $M$ of $K_{2n}$ satisfying 
    \[ f_c(M) = \bigO(k^2). \]
\end{theorem}

In fact, our proof techniques allow us to obtain much stronger results for representative subgraphs in the context of vector-labelled graphs. For a graph $G$, a subgraph $G'$ of $G$, and a function $\vecf\colon E(G)\to\R^k$ assigning vectors to the edges of $G$, we denote by $\vecf(G')$ the sum $\sum_{e \in E(G')} \vecf(e)$. We say that $G'$ is \defn{representative} of $G$ if $\vecf(G')/e(G')$ is equal to $\vecf(G)/e(G)$. Analogously to the definition of $f_c$, we define
$$f_\vecf(G') = \left\|\vecf(G') - \frac{e(G')}{e(G)}\vecf(G)\right\|_1.$$
Thus $f_\vecf$ measures how far from representative a subgraph $G'$ of $G$ is. Note that by translating each vector by a constant amount so that $\vecf(G) = 0$, we may equivalently view the problem of finding a representative subgraph as a problem in zero-sum Ramsey theory over $\mathbb{R}^k$. The following theorem shows that, under mild assumptions on $\vecf$, complete graphs admit almost representative perfect matchings.
\begin{restatable}{theorem}{vecComplete}\label{th:vec-complete}
   Let $\vecf \colon E(K_{2n}) \to \mathbb{R}^k$ such that $\norm{\vecf(e)}_1\leq 1$ for all $e\in E(K_{2n})$. Then there is a perfect matching $M$ of $K_{2n}$ satisfying 
   \[
        f_\vecf(M) = \bigO(k^2).
   \]
\end{restatable}
To see that this is indeed a more general setting than the edge-colouring problem, let $b_1,\dotsc,b_k$ denote the standard basis of $\R^k$. Given a $k$-edge-coloured graph $G$, let $\vecf$ assign each edge $e$ of colour $i$ the vector $b_i$. Then when $G$ is colour-balanced, a perfect matching $M$ of $G$ satisfies $f_\vecf(M) = f_c(M)$, and in general, a subgraph is representative exactly when it contains the expected proportion of the edges in each colour class. Furthermore, when $G$ is not colour-balanced, our theorem gives the same quadratic error bound for representative perfect matchings of graphs with unbalanced colour-counts. 

Various authors have also considered bounds for other colour-balanced subgraphs. Hollom \cite{hollom2025uniform} asked whether, for an $r$-vertex graph $F$, it is always possible to find an $F$-factor $H$ in a colour-balanced $K_n$ such that $f_c(H)$ is bounded by a function of only $F$ and $k$. This question was answered affirmatively by Banerjee and Hollom \cite{banerjee2026}, who obtained a bound of $f_c(H) \leq (8rk)^{(8rk)^{k}}$. When $k=2$, there has also been considerable research on embedding spanning forests. For any fixed spanning forest $F$, it is known that $f_c(F)$ is always bounded by an error that is linear in $\Delta$, with successive improvements to the constant achieved by Mohr, Pardey and Rautenbach \cite{mohr2022zero} and Pardey and Rautenbach \cite{pardey2023efficiently}. Hollom, Mond and Portier \cite{hollom2024} achieved the current best bound of $\frac{\Delta}{2} + 18$, which is tight up to the additive constant. Furthermore, when the restriction to a fixed isomorphism class is dropped, Caro, Hansberg, Lauri and Zarb \cite{caro2022zero} showed that when $k=2$, every colour-balanced complete graph admits a colour-balanced spanning tree.

Our main result is a bound for a far more general class of graphs, which generalises and significantly improves upon the previous known bounds for balanced subgraphs with $k$ colours. The following theorem applies to arbitrary spanning subgraphs with bounded maximum degree in our more general vector setting, and fully resolves a question of Banerjee and Hollom \cite{banerjee2026}. 

\begin{restatable}{theorem}{coolbounds}\label{thm:bounded_deg_bound}
    Let $H$ be an $n$-vertex graph with maximum degree $\Delta$, and let ${\vecf : E(K_n) \to \mathbb{R}^k}$ be such that $\|\vecf(e)\|_1 \leq 1$ for all $e \in E(K_n)$. Then there is a copy $H'$ of $H$ in $K_n$ satisfying
    $$f_\vecf(H') = \bigO(\Delta^2k^2\sqrt{\log2\Delta}).$$
\end{restatable}
For certain graphs $H$, we are also able to improve this bound. When $H$ is an $F$-factor for some graph $F$ on $r$ vertices, we obtain an error bound of $\bigO(\Delta r k^2)$, which is at worst $\bigO(r^2 k^2)$, further improving the recent $(8rk)^{(8rk)^{k}}$ bound of Banerjee and Hollom. We also improve the bound in the case where $H$ is a spanning forest. We consider spanning forests in our more general vector setting with arbitrary $k \geq 2$ and obtain an error of $\bigO(\Delta k^2)$, showing that the error remains linear in $\Delta$ when $k \geq 3$. This is the first bound for forests with more than $2$ colours.
\begin{restatable}{theorem}{forestError}\label{thm:forest_error}
    Let $F$ be an $n$-vertex forest with maximum degree $\Delta$, and let $\vecf : E(K_n) \to \mathbb{R}^k$ such that $\|\vecf(e)\|_1 \leq 1$ for all $e \in E(K_n)$. Then there is a copy $F'$ of $F$ in $K_n$ satisfying 
    \[f_\vecf(F') = \bigO(\Delta k^2).\]
\end{restatable}
We also give an extension to colourings of hypergraphs: we obtain a quadratic error-bound for perfect matchings of the $rn$-vertex $r$-uniform complete hypergraph $K_{rn}^{(r)}$.  
\begin{restatable}{theorem}{hypergraph}\label{thm:hypergraph_matching}
    Let $\vecf \colon E(K_{rn}^{(r)}) \to \mathbb{R}^k$ such that $\| \vecf(e) \|_1 \leq 1$ for all $e \in E(K_{rn}^{(r)})$. Then there is a perfect matching $M$ of $K_{rn}^{(r)}$ satisfying
   \[
        f_\vecf(M) = \bigO_r(k^2).
   \]
\end{restatable}

Furthermore, we consider the existence of colour-balanced spanning trees in complete graphs, when the restriction to a fixed isomorphism class for the tree is dropped. Generalising a result of Caro, Hansberg, Lauri and Zarb~\cite{caro2022zero} for $2$-edge-coloured complete graphs, we use a matroid-theoretic argument to determine an exact condition for when $k$-edge-coloured complete graphs admit colour-balanced spanning trees, in terms of the number of edges in each colour class. 
\begin{restatable}{theorem}{colourBalancedTree}\label{thm:colour_balanced_tree}
    Let $t \geq 1$ and let $c$ be a $k$-edge-colouring of $K_{2kt+1}$. Suppose that, for all $j \in [k]$ and every set of $j$ colours in $[k]$, there are strictly more than ${2jt \choose 2}$ edges labelled by those $j$ colours. Then there exists a colour-balanced spanning tree of $K_{2kt+1}$.
\end{restatable}
A corollary of this result is that every colour-balanced complete graph admits a colour-balanced spanning tree, up to divisibility constraints. 
\begin{corollary}\label{cor:balanced_tree}
    Let $t \geq 1$ and let $c$ be a colour-balanced $k$-edge-colouring of $K_{2kt+1}$. Then there is a colour-balanced spanning tree of $K_{2kt+1}$.
\end{corollary}

Finally, we consider lower bounds. We prove lower bounds for $f_c$ in the complete and complete bipartite perfect matching settings. While Pardey and Rautenbach found an example of a colour-balanced $K_6$ with $k=3$ that does not admit a colour-balanced perfect matching, this appears to be the only previously known counterexample to the existence of colour-balanced perfect matchings in complete graphs. We exhibit an infinite class of counterexamples, establishing that colour-balanced $k$-edge-colourings of $K_{2kt}$ are not guaranteed to admit colour-balanced perfect matchings, even for very large $t$. We also establish that colour-balanced $k$-edge-colourings of $K_{kt,kt}$ are not guaranteed to admit colour-balanced perfect matchings: in fact, for infinitely many pairs $(k,t)$, there exist edge-colourings of $K_{kt,kt}$ where every perfect matching $M$ satisfies $f_c(M) \geq \sqrt{k/2}$. This suggests that there is a hard barrier to significantly improving our upper bounds with similar techniques. 

\subsection{Outline}

A central tool, which we use throughout the paper, is a theorem on representative perfect matchings in the bipartite graph $K_{n,n}$. We begin in \cref{sec:knn_section} by proving this result.
\begin{restatable}{theorem}{vecBipartite}\label{thm:vec-bipartite}
    Let $\vecf \colon E(K_{n,n}) \to \mathbb{R}^k$ be such that $\| \vecf(e) \|_1 \leq 1$ for all $e \in E(K_{n,n})$. Then there is a perfect matching $M$ of $K_{n,n}$ satisfying
    \[
        f_\vecf(M) = \bigO(k^2).
    \]
\end{restatable}
While complete bipartite graphs do not seem to have been considered in the broader literature on representative subgraphs, this result is the cornerstone of our remaining results. In particular, it allows us to construct representative bijections between sets, which we will leverage to construct representative embeddings of more complex subgraphs $H$ in $K_n$. In particular, this paper establishes a variety of strategies for applying \cref{thm:vec-bipartite} to obtain other bounds for representative subgraphs, using a variety of different approaches for partitioning the vertex sets of $H$ and $K_n$.

In \cref{sec:other_graphs}, we use \cref{thm:vec-bipartite} to prove our main results for representative matchings in $K_{2n}$ and general representative spanning subgraphs of bounded degree, as well as for representative matchings of uniform hypergraphs. In all cases, our strategy is broadly the same. We begin with a complete (hyper)graph, and take a random partition of the vertex set into some number of parts. With positive probability, the induced multi-partite (hyper)graph is almost representative of the original complete graph. We then apply \cref{thm:vec-bipartite} to embed our graph in this subgraph, with the precise embedding strategy changing in each context. In the simplest case of matchings in $K_{2n}$, it suffices to find an almost representative $K_{n,n}$ subgraph of $K_{2n}$. For arbitrary spanning graphs $H$, we partition $V(H)$ into independent sets, and then aim to use \cref{thm:vec-bipartite} to embed one independent set at a time into a suitable partition of $V(K_n)$. Suitable bounds for embedding $H$ can be obtained in this manner as long as there is a suitably `uniform' way to partition $V(H)$ into independent sets. By considering properties of random vertex colourings, we show that bounded maximum degree graphs admit suitable uniform partitions. 

In \cref{sec:spanning_forests}, with some additional work, we obtain an improved `uniform' partition for spanning subgraphs that are forests. Applying the embedding result from \cref{sec:other_graphs} with this improved uniformity bound proves \cref{thm:forest_error}. We also characterise the edge-colourings of complete graphs that admit some colour-balanced spanning tree, by reduction to a matroid optimisation problem. Finally, in \cref{sec:lower-bounds} we present a series of constructions that establish lower bounds for the matchings problem on complete and complete bipartite graphs. We conclude in \cref{sec:open_problems} with a discussion of some open problems.

\textbf{Notation and terminology:}
Throughout the paper, we use the following notation. We denote by $[k]$ the set $\{1,\dots,k\}$. For $x,y\in\R^k$, let $x^{(i)}$ denote the $i$th coordinate of $x$, let $\norm{x}_p$ be the $\ell_p$ norm of $x$, and let $x\cdot y$ or simply $xy$ denote the dot product of $x$ and $y$. We also sometimes write $x^2$ for $x\cdot x$. For any graph $G$, we denote by $V(G)$ and $E(G)$ the vertex set and edge set of $G$, respectively, and by $e(G)$ the size of $E(G)$. For (not necessarily disjoint) subsets $A$ and $B$ of $V(G)$, we denote by $E(A,B)$ the set of edges in $G$ with one endpoint in $A$ and the other in $B$, and by $e(A,B)$ the size of $E(A,B)$. 

\section{Upper bound for $K_{n,n}$}\label{sec:knn_section}

In this section, we prove \cref{thm:vec-bipartite}. We begin by outlining the argument.

\textbf{Step 1: Linear relaxation.}
We will first consider a linear relaxation of the problem. A perfect matching of $K_{n,n}$ can be viewed as an edge-weighting of $K_{n,n}$ in which every edge weight is either $0$ or $1$, and the sum of the weights of edges incident with any vertex is precisely $1$. Allowing the edge weights to instead be in $[0,1]$ with the same vertex constraints, we obtain a `fractional perfect matching' of $K_{n,n}$. To construct an almost representative perfect matching, we will first find a fractional perfect matching in which all but a few edges are weighted $0$ or $1$, and the weighted sum of the edge labels is representative. 

\textbf{Step 2: Matching the remaining edges.} We will show that after Step 1, the edges with non-integral weight form $\bigO(k)$ edge-disjoint paths and a small number of additional edges. After deleting these additional edges, the remainder of the problem reduces to finding an appropriate matching on each of the edge-disjoint paths. We aim to split each path into segments, such that in each segment we include either all the even edges or all the odd edges in our matching. Ensuring that this split achieves an almost representative matching with respect to the path edge weights and edge labels is done using a necklace-splitting type theorem. As a result, we make $\bigO(k)$ ``cuts'' on each of the $\bigO(k)$ paths, and we incur an $\bigO(k^2)$ error in our overall matching by arbitrarily matching pairs of unmatched leftover vertices.

\subsection{Linear relaxation}

We need to construct a representative fractional matching of $K_{n,n}$ with few fractional weight edges. Our strategy is to start with every edge weighted $\frac{1}{n}$ and iteratively redistribute weights around certain cycles to increase the number of integer-weighted edges while maintaining that the fractional matching is representative. We continue this process until the remaining fractional subgraph of $K_{n,n}$ is very close to a forest. 

Let $F$ be a spanning forest of a graph $G$, which will later consist of the fractional edges of $K_{n,n}$. For each edge $e \in E(G) \setminus E(F)$, there is a unique cycle $C_e$ contained in $F \cup \{e\}$, and the set of all such cycles is called the \defn{fundamental cycle basis} of $G$. The size of the fundamental cycle basis is known as the \defn{cyclomatic number} of $G$, and is equal to $|E(G)| - |V(G)| + c(G)$, where $c(G)$ is the number of connected components of $G$. Formally, we show that we can appropriately redistribute weights around cycles in the fundamental cycle basis as long as the cyclomatic number of the fractional subgraph remains sufficiently large.

\begin{lemma}\label{lem:linear_relaxation}
    Let $\vecf \colon E(K_{n,n}) \to \mathbb{R}^k$ such that $\| \vecf(e) \|_1 \leq 1$ for all $e \in E(K_{n,n})$. Then there is an edge-weighting $w$ of $K_{n,n}$ satisfying $\sum_{e} w(e)\vecf(e)= \frac{1}{n}\cdot \vecf(K_{n,n})$ and $0 \leq w(e) \leq 1$ for every edge $e$, and for which the subgraph of $K_{n,n}$ consisting of edges with fractional weight has cyclomatic number at most $k$.
\end{lemma}
\begin{proof}
    Let $w$ be an edge-weighting of $K_{n,n}$ such that $\sum_{e} w(e)\vecf(e)= \frac{1}{n}\vecf(K_{n,n})$, for each edge $e$ we have $0 \leq w(e) \leq 1$, and for each vertex $v$ we have $\sum_{e \ni v} w(e) = 1$. Such a $w$ exists, since $w(e) = \frac{1}{n}$ for every $e$ suffices. We choose $w$ such that the number of edges with non-integer weight in $K_{n,n}$ is minimised, and let $H$ be the (spanning) subgraph of $K_{n,n}$ consisting of the edges with non-integer weights. 
    We show that $H$ is close to a forest.
    
    Suppose that a fundamental cycle basis $S$ of $H$ has size at least $k+1$, and let $C = e_1,e_2,\dots,e_{2\ell}$ be a cycle in $S$ of length $2\ell$. For $a \in \R$, define $w_{C,a}$ to be a function that updates the current weighting $w$ on $E(K_{n,n})$, by keeping edge weights of edges not in $C$ constant, and redistributing the edge weights in $C$ as follows:
    \begin{align*}
        w(e_j) \shortleftarrow
        \begin{cases}
		w(e_j) + \varepsilon a & \text{if $j$ is even,}\\
            w(e_j) - \varepsilon a & \text{if $j$ is odd,}
	\end{cases}
    \end{align*}
    for each $j \in [2\ell]$. Note that updating the weights on $C$ by $w_{C,a}$ preserves the property that $\sum_{e \ni v} w(e) = 1$ at every vertex $v$. Furthermore, applying $w_{C,a}$ changes the value of $\sum_{e} w(e)\vecf(e)$ by the vector 
    $u_C = \varepsilon \sum_{j=1}^{2\ell} \vecf(e_j)(-1)^j\in\R^k$.

    Since $S$ has size at least $k+1$, there is a linear dependency among the set of vectors $\{u_C\}_{C\in S}$, say $\sum_{C\in S}a_C u_C=0$, with real values $a_C$ not all zero. For each $C\in S$, modify the weights of edges in $C$ by $w_{C,a_C}$. Since each cycle in the fundamental cycle basis contains an edge that is contained in no other fundamental cycle, it follows that at least one edge in $H$ now has a new weight. By choosing $\varepsilon>0$ minimal such that this action results in at least one edge in $H$ being assigned the weight either $0$ or $1$, it follows that $H$ now has strictly fewer fractionally weighted edges, and every edge $e$ in $H$ still has a weight $0 \leq w(e) \leq 1$. Our choice of vectors necessarily maintains $\sum_{e} w(e)\vecf(e)= \frac{1}{n}\vecf(K_{n,n})$; furthermore, $K_{n,n}$ still satisfies $\sum_{e \ni v} w(e) = 1$ at every vertex $v$. It follows that our choice of $H$ was not minimal with respect to the number of fractionally weighted edges in $K_{n,n}$. Hence, we may assume that a fundamental cycle basis of $H$ contains at most $k$ cycles, and so by definition, $H$ has cyclomatic number at most $k$, as required.
\end{proof}

\subsection{Matching the leftover edges}

In this section we show that we can obtain a representative matching of the fractionally weighted graph obtained in \cref{lem:linear_relaxation} by applying a generalised version of Alon's well-known necklace splitting theorem \cite{alon1987}. We begin by establishing the version of the theorem that we will use. 
A \defn{charge} $\rho$ on $\R^d$ is a signed finite Borel measure that can be written as $\rho=\mu^{+}-\mu^{-}$, where $\mu^{\pm}$ are finite Borel measures satisfying $\mu^{\pm}(P)=0$ for every hyperplane $P$. By a \defn{half-space}, we mean either of the two parts into which a hyperplane partitions $\mathbb{R}^d$. We also allow a half-space to be empty or equal to $\mathbb{R}^d$. Akopyan and Karasev~\cite{Akopyan_2012} showed that the classical proof of the ham sandwich theorem (see Matou{\v{s}}ek~\cite{Matousek2003UsingBorsukUlam}) extends directly to charges as follows.\\
\begin{theorem}[Ham sandwich theorem for charges, \cite{Akopyan_2012}]\label{thm:ham_sandwich_charges}
    Let $\rho_1, \ldots , \rho_d$ be $d$ charges in $\R^d$. Then there exists a half-space $H$ such that, for all $i \in [d]$,
\[
    \rho_i(H)=\frac{1}{2}\cdot\rho_i(\mathbb{R}^d).
\]
\end{theorem}
We will need a necklace splitting theorem for charges. A theorem of this type for probability measures was proved by Stromquist and Woodall in~\cite{stromquist1985sets}. However, by applying \cref{thm:ham_sandwich_charges} in place of the traditional ham sandwich theorem in their proof, one can obtain the following, more general statement. We provide this slightly amended version of the proof from~\cite{stromquist1985sets} for completeness. 
\begin{theorem} \label{thm:necklaceGeneral}
    Let $d \geq 2$, and let $\rho_1,\dots,\rho_d$ be charges on a segment $I\subset\R$. For each $\alpha \in [0,1]$, there is a set $K_\alpha \subseteq I$ such that $K_\alpha$ is a union of at most $d+1$ intervals, and $\rho_i(K_\alpha) = \alpha\cdot \rho_i(I)$ for each $i \in [d]$.
\end{theorem}

\begin{proof}
    Let $X$ be the set of $\alpha\in [0,1]$ for which there exists a set $K_\alpha\subset I$ that is a union of intervals with at most $2d$ endpoints in the interior of $I$ and satisfies $\rho_i(K_\alpha)=\alpha\cdot\rho_i(I)$ for all $i\in[d]$. We will show that $X=[0,1]$, from which the theorem clearly follows. Noting that $\rho_i(x) = 0$ for each point $x \in I$ and each $i \in [d]$, a standard sequential compactness argument shows that $X$ is closed. Therefore, it suffices to show that $X$ is dense in $[0,1]$. We prove this by establishing:
    \begin{enumerate}[label=(\roman*), noitemsep]
        \item $X$ contains $0$, and if $\alpha$ is in $X$, then so is $1-\alpha$; \label{enum:gennecklace-first}
        \item if $\alpha$ is in $X$, then so is $\frac{1}{2}\cdot\alpha$. \label{enum:gennecklace-second}
    \end{enumerate} 
    For~\ref{enum:gennecklace-first}, observe that $0\in X$ is witnessed by $K_0=\varnothing$ and that $K_{1-\alpha}\coloneqq I\setminus K_{\alpha}$ witnesses $1-\alpha\in X$ when $\alpha\in X$, as the number of internal endpoints remains the same after passing to the complement. It remains to show~\ref{enum:gennecklace-second}. Fix $\alpha\in X$ and a corresponding set $K_\alpha$. Define the moment curve $\gamma\colon\R\to\R^d$ by $\gamma(t)=(t,t^2,\dots,t^d)$, and for each $i\in[d]$, let $\eta_i$ be the signed Borel measure on $\R^d$ whose value on a Borel set $\Omega\subseteq\R^d$ is given by 
    \[
        \eta_i(\Omega)=\rho_i(\gamma^{-1}(\Omega)\cap K_\alpha).
    \] 
    Any hyperplane intersects with the moment curve in at most $d$ points and hence has zero $\eta_i$-measure, and so each $\eta_i$ is a charge on $\R^d$. Applying \cref{thm:ham_sandwich_charges} to $\eta_1,\dots,\eta_d$, we obtain a half-space $H\subseteq\R^d$ with $\eta_i(H)=\frac{1}{2}\cdot \eta_i(\R^d)$ for all $i\in[d]$. Define ${K'\coloneq\gamma^{-1}(H)\cap K_\alpha}$ and ${K''\coloneq K_\alpha\setminus K'}$. Both $K'$ and $K''$ are collections of intervals, and for each $i\in[d]$ they satisfy
    \begin{align*}
        \rho_i(K') &= \rho_i(\gamma^{-1}(H)\cap K_\alpha)=\eta_i(H)=\frac{1}{2}\cdot\eta_i(\R^d)=\frac{1}{2}\cdot\rho_i(K_\alpha)=\frac{1}{2}\alpha\cdot\rho_i(I), \\
        \rho_i(K'') &= \rho_i(K_\alpha \setminus K')=\frac{1}{2}\cdot\rho_i(K_\alpha)=\frac{1}{2}\alpha\cdot\rho_i(I).
    \end{align*}
     The set $\gamma^{-1}(H)\cap I$ is a union of intervals with at most $d$ interior endpoints, since either $H$ is bounded by a hyperplane that intersects the moment curve in at most $d$ points, or $H$ is empty or $\R^d$ and the statement holds trivially. The original set $K_\alpha$ has at most $2d$ interior endpoints by assumption, and each interior endpoint of $K'$ and $K''$ arises either from an endpoint of $K_\alpha$ or from an endpoint of $\gamma^{-1}(H)\cap I$. Consequently, the total number of interior endpoints of $K'$ and $K''$ together is at most $2d+2d=4d$, so one of $K'$ and $K''$ has at most $2d$ interior endpoints and thus witnesses that $\frac{\alpha}{2}\in X$, completing the proof.
\end{proof}
We are now ready to apply this generalised necklace splitting theorem to obtain representative matchings of edge-weighted paths.

\begin{lemma}\label{lem:path_bound}
    Let $P_n = e_1,\dots,e_{n-1}$ be an $n$-vertex path, let $\vecf : E(P_n) \to \mathbb{R}^k$ satisfy $\| \vecf(e) \|_1 \leq 1$ for all $e \in E(P_n)$, and let $\alpha \in [0,1]$. Let $w : E(P_n) \to [0,1]$ be defined by 
    \[w(e_i) = \begin{cases}
        \alpha & i \text{ odd} \\
        1 - \alpha & i \text{ even}.
    \end{cases}\]
    Then there is a matching $M$ of $P_n$ of size $\frac{n}{2} - \bigO(k)$ such that 
    \[
        \left\| \vecf(M) - \sum_{e \in P_n}w(e)\vecf(e) \right\|_1  = \bigO(k).
    \] 
\end{lemma}
\begin{proof}
    \newcommand{\sh}[0]{H} 
    When $i$ is odd, we call $e_i$ an \defn{odd edge} and when $i$ is even we call $e_i$ an \defn{even edge}. Note that the odd edges of $P_n$ all have weight $\alpha$, and the even edges all have weight $1 - \alpha$, so $\sum_{e \in P_n}w(e)\vecf(e)=\alpha \sh_O + (1-\alpha)\sh_E$, where $\sh_O$ and $\sh_E$ denote the sums of $\vecf(e)$ over all odd and all even edges, respectively. 
    We will split $E(P_n)$ into two classes, such that the sum $\sum \vecf(e)$ over the odd (resp. even) edges in the first class is close to $\alpha \sh_O$ (resp. $\alpha \sh_E$). In other words, the first class will receive an $\alpha$ proportion of both $\sh_O$ and $\sh_E$. We will then obtain the desired $M$ by matching the odd edges from the first class and the even edges from the second class. 

    First, identify $P_n$ with the interval $I_n=[0,n)$ by identifying each edge $e_j$ in $P_n$ with the interval $[j-1, j)$. Let $O$ (resp. $E$) denote the subintervals of $I_n$ identified with odd (resp. even) edges of $P_n$, and for each $i \in [k]$, define two charges $\rho_i^O$ and $\rho_i^E$ on $I_n$ as follows. The density of $\rho_i^O$ on $[j-1, j)$ is equal to $\vecf(e_j)^{(i)}$ if $j$ is odd and $0$ otherwise, and the density of $\rho_i^E$ on $[j-1, j)$ is equal to $\vecf(e_j)^{(i)}$ if $j$ is even and $0$ otherwise. It follows from \cref{thm:necklaceGeneral} that there is a set $K_{\alpha} \subseteq I_n$ that is a union of at most $2k+1$ intervals in $I_n$ and satisfies $\rho_i^O(K_\alpha) = \alpha \rho_i^O(I_n) = \alpha \sum_{e \in O}\vecf(e)^{(i)}$ and $\rho_i^E(K_\alpha) = \alpha \rho_i^E(I_n) = \alpha \sum_{e \in E}\vecf(e)^{(i)}$ for each $i \in [k]$. 

    We now restrict our attention to edges of $P_n$ whose associated subintervals in $I_n$ do not contain one of the at most $4k+2$ endpoints of the intervals in $K_\alpha$. Each of these edges is identified with a subinterval that is either contained in the interior of $K_\alpha$, or the interior of
    $K_\alpha^c=I_n\setminus K_\alpha$. For a subset $A \subseteq I_n$, we write $e \in A$ if the subinterval of $I_n$ identified with $e$ is contained in the interior of $A$. Note that for each $i \in [k]$, the support of $\rho_i^O$ in $K_\alpha$ is contained in $K_\alpha \cap O$, and the support of $\rho_i^E$ in $K_\alpha^c$ is contained in $K_\alpha^c \cap E$. Let $\xi$ denote the sum of the vector labels on the at most $4k+2$ edges of $P_n$ that are not contained in the interior of either $K_\alpha$ or $K_\alpha^c$. Then the following inequality holds for each $i \in [k]$.
    \begin{equation}
        \Big\lvert\sum_{e \in K_\alpha \cap O} \left(\rho_i^O(e)\right) + \sum_{e \in K_\alpha^c \cap E} \left(\rho_i^E(e)\right) - \left(\rho_i^O(K_\alpha) + \rho_i^E(K_\alpha^c)\right)\Big\rvert \leq \lvert\xi^{(i)}\rvert
    \end{equation}
    Now, let $M$ be the matching consisting of the edges in $K_\alpha \cap O$ and in $K_\alpha^c \cap E$. It follows that for each $i \in [k]$, we have
    \begin{align*}
        \Big\lvert\sum_{e\in M} \vecf(e)^{(i)} -\sum_{e \in P_n}w(e)\vecf(e)^{(i)}\Big\rvert &= \Big\lvert\sum_{e\in K_\alpha \cap O} \vecf(e)^{(i)} + \sum_{e\in K_\alpha^c \cap E} \vecf(e)^{(i)} - \sum_{e \in P_n}w(e)\vecf(e)^{(i)}\Big\rvert \\
        &= \Big\lvert\sum_{e\in K_\alpha\cap O}\rho_i^O(e) + \sum_{e\in K_\alpha^c\cap E}\rho_i^E(e) - \sum_{e\in P_n} w(e)\vecf(e)^{(i)} \Big\rvert\\ 
        &\leq \Big\lvert\rho_i^O(K_\alpha) + \rho_i^E(K_\alpha^c) -\sum_{e\in P_n} w(e)\vecf(e)^{(i)} \Big\rvert + \lvert\xi^{(i)}\rvert \\
        &= \Big\lvert\alpha\sum_{e\in O}\vecf(e)^{(i)} + (1-\alpha)\sum_{e\in E}\vecf(e)^{(i)} - \sum_{e\in P_n} w(e)\vecf(e)^{(i)} \Big\rvert + \lvert\xi^{(i)}\rvert \\ 
        &= \lvert\xi^{(i)}\rvert.
    \end{align*}
    Since each of the at most $4k+2$ edges of $P_n$ that are not contained in the interior of either $K_\alpha$ or $K_\alpha^c$ satisfies $\|\vecf(e)\|_1 \leq 1$, it follows that $\|\xi\|_1 \leq 4k+2=\bigO(k)$, and so we have achieved the desired error term. Finally, it follows from the construction of $M$ that each of the at most $2k+1$ path segments of $P_n$ contained strictly in the interior of either $K_\alpha$ or $K_\alpha^c$ has at most $2$ vertices not incident with $M$. It follows that $|M| \geq \frac{n}{2} - 2k - 1$. This completes the proof.
\end{proof}

Finally, we prove our result on the complete bipartite graph. 
\vecBipartite*
\begin{proof}
    By \cref{lem:linear_relaxation}, there is an edge-weighting $w$ of $K_{n,n}$ satisfying $\sum_e w(e) h(e) = \frac{1}{n}\vecf(K_{n,n})$ and $0 \leq w(e) \leq 1$ at every edge, and for which the subgraph $H$ of $K_{n,n}$ consisting of edges with fractional weights has cyclomatic number at most $k$. 

    Since $H$ has cyclomatic number at most $k$, we have that $|E(H)| \leq k + |V(H)| - 1$. Take a partial matching $M$ of $K_{n,n}$ consisting of every edge in $K_{n,n}$ with edge-weight $1$, and note that $H$ is vertex disjoint from this matching. Furthermore, since $\sum_{e \ni v} w(e) = 1$ at every vertex $v$ in $K_{n,n}$, it follows that $\sum_{e \ni v} w(e) = 1$ at every vertex $v$ in $H$, and hence that $H$ has no vertices of degree $1$. It follows that the average degree of vertices in $H$ is at least $2$, and since $|E(H)| \leq k + |V(H)| -1$, we have that $H$ has at most $2k$ vertices with degree greater than $2$, and is at most $2k$ edge-deletions away from a graph with maximum degree $2$. By deleting all but one edge at each vertex in $H$ with degree strictly greater than $2$, we remove at most $4k$ edges from $H$ to obtain a graph in which every non-leaf vertex had degree $2$ in $H$. By deleting at most a further $k$ edges, we obtain a disjoint union of paths $F$ with this property. Note that $F$ was obtained from $H$ by deleting at most $5k$ edges, and so $F$ is a union of at most $5k$ paths, $P_1,\dots,P_{5k}$. Since every vertex $v$ in $H$ satisfies $\sum_{e \ni v} w(e) = 1$, it follows that the edge weights on each component $P_i$ of $F$ alternate between $\alpha_i$ and $1-\alpha_i$ along the path, for some $0 < \alpha_i < 1$. 

    Applying \cref{lem:path_bound} to each of the $\bigO(k)$ paths in $F$, we obtain a matching $M'$ of $F$ of size at least $\frac{|V(F)|}{2} - \bigO(k^2)$ satisfying $\|\vecf(M') - \sum_{e \in F}w(e)\vecf(e) \|_1 = \bigO(k^2)$. Since we deleted at most $5k$ edges of weight at most $1$ from $H$ to obtain $F$, it follows that $\|\vecf(M') - \sum_{e \in H}w(e)\vecf(e) \|_1 = \bigO(k^2)$. Observe that the union of $M'$ and $M$ is a matching of $K_{n,n}$. Furthermore, $M$ is a perfect matching of every vertex in $K_{n,n}$ that is not in $H$, and we deleted at most $5k$ edges of weight at most $1$ from $H$ to obtain $F$. So, since $M'$ is at most $\bigO(k^2)$ edges smaller than a perfect matching of $F$, we may obtain from $M' \cup M$ a perfect matching of $K_{n,n}$ by arbitrarily matching the remaining unmatched vertices with at most an additional $\bigO(k^2)$ edges, each of which satisfies $\|\vecf(e)\|_1 \leq 1$. The theorem statement now follows. 
\end{proof}

\section{Embedding bounded-degree (hyper)graphs}\label{sec:other_graphs}

In this section we use \cref{thm:vec-bipartite} to construct almost representative subgraphs in a range of other contexts. We begin with perfect matchings of $K_{2n}$, proving \cref{th:vec-complete}. Next, we generalise this result to obtain bounds for representative perfect matchings of complete $r$-uniform hypergraphs. Finally, we prove \cref{thm:arbitrary-embedding}, for embeddings of bounded-degree spanning graphs in the complete graph. In all cases, our strategy is broadly the same. We begin with a complete (hyper)graph, and take a random partition of the vertex set into some number of parts. With positive probability, the induced multi-partite (hyper)graph is almost representative of the original complete graph. We then apply \cref{thm:vec-bipartite} to embed our graph in this subgraph, with the precise embedding strategy changing in each context.

\subsection{Perfect matchings of complete graphs}

In order to apply \cref{thm:vec-bipartite} to prove \cref{th:vec-complete}, we need to pass from $K_{2n}$ to a $K_{n,n}$ subgraph that is close to representative of the original $K_{2n}$. A straightforward second moment argument shows that this is possible. 

\begin{lemma}\label{lem:vec-partition}
    Let $K_{2n} = (V,E)$, and suppose that $\vecf \colon E \to \R^k$ satisfies $\norm{\vecf(e)}_1\leq 1$ for all $e \in E$. Then $K_{2n}$ contains a complete bipartite subgraph $K_{n,n}$ such that
    \[
        \frac{1}{n}\cdot f_\vecf(K_{n,n}) = \bigO(\sqrt{k}).
    \]
\end{lemma}
\begin{proof}
    Note that $f_\vecf$ is invariant under shifting $\vecf$ by a constant vector, and scales linearly if $\vecf$ is scaled by a constant. Replace $\vecf$ with $\frac{1}{2}(\vecf-\bar\vecf)$, where $\bar\vecf$ is the average value of $\vecf(e)$. It follows that $\vecf(K_{2n})=0$ and $\norm{\vecf(e)}_1\leq 1$ for all $e\in E$. We show that, under these assumptions, $K_{2n}$ contains a $K_{n,n}$ subgraph such that $\frac{1}{n}f_\vecf(K_{n,n}) = \|\frac{1}{n}\vecf(K_{n,n})\|_1=\bigO(\sqrt{k})$. Consider a random equipartition $(A,B)$ of $V$ and let $K_{n,n}$ have vertex parts $A$ and $B$.
    
    For conciseness, we denote $\vecf(e)$ by $\vecf_e$ and for each $v \in V$, let $d_v=\sum_{e\ni v}\vecf_e$. 
    For each $e\in E$, define $X_e$ to be equal to $\vecf_e$ if $e\in E(A,B)$ and $0$ otherwise, so that ${X\coloneq\sum_{e\in E}X_e=\vecf(K_{n,n})}$ is the quantity we are aiming to minimise. 
    Let $p=\frac{n}{2n-1}$ denote the probability that a fixed edge $e$ has endpoints in both $A$ and $B$, and for $s\in\{0,1,2\}$ let $p_s$ be the probability that an edge $e'$ has endpoints in both $A$ and $B$, given that $e$ does, and that $|e'\cap e|=s$. Thus $p_0=\frac{n-1}{2n-3}$ for disjoint $e$ and $e'$, $p_1=\frac{1}{2}$ for incident edges and $p_2=1$ for $e'=e$.
    For any $e,e'\in E$ with $|e\cap e'|=s$, we have that
    \[
        \E X_e X_{e'} = p\cdot p_s\cdot \vecf_e \vecf_{e'}.
    \]
    Since $\sum_{e,e'}h_eh_{e'}=\left(\sum_e h_e\right)^2=0$ and  $h_e^2\leq\|h_e\|_1\leq 1$ for all $e\in E$, it follows that 
    \begin{align*}
        \E \|X\|_2^2 = \E\left[\sum_{e,e'\in E}X_e X_{e'}\right] &= p\left(p_0\sum_{e,e'\text{ disj.}}\vecf_e \vecf_{e'} + p_1\sum_{e,e'\text{ inc.}}\vecf_e \vecf_{e'} + p_2\sum_{e\in E}\vecf_e^2\right)\\
        &= p\left(p_0\sum_{e,e'\in E}\vecf_e\vecf_{e'}+(p_1-p_0)\sum_{v\in V}d_v^2+(p_0-2p_1+p_2)\sum_{e\in E}\vecf_e^2\right) \\ &\leq p (p_2 -2p_1+p_0)\sum_{e\in E}\vecf_e^2 \leq p\cdot p_0\cdot|E|<n^2.
    \end{align*}
    Hence, there exists an equipartition of $V$ such that $\|X\|_2 < n$, and thus a $K_{n,n}$ subgraph of $K_{2n}$ such that $f_\vecf(K_{n,n})=\|X\|_1\leq\sqrt{k}\|X\|_2<\sqrt{k}\cdot n$, as desired. \qedhere
\end{proof}

\vecComplete*
\begin{proof}
    Apply \cref{lem:vec-partition} to find a $K_{n,n}$ subgraph of $K_{2n}$ such that $\frac{1}{n}f_\vecf(K_{n,n}) = \bigO(\sqrt{k})$.
    Applying \cref{thm:vec-bipartite} to $K_{n,n}$, we obtain a perfect matching $M$ with the required property.
\end{proof}
\cref{thm:complete_colours} now follows from \cref{th:vec-complete} by assigning the basis vector $b_i$ to each edge of colour $i$.

\subsection{Perfect matchings of complete hypergraphs}

We now generalise our approach to obtain bounds for representative perfect matchings in complete hypergraphs. The $r$-uniform \defn{complete hypergraph} $K_n^{(r)}$ is the $n$-vertex hypergraph whose hyperedge set consists of all $r$-vertex subsets of $V(K_n^{(r)})$. An $r$-uniform \defn{$r$-partite} hypergraph is a hypergraph with $r$ vertex parts such that every edge contains exactly one vertex from each of the $r$ parts. We call an $r$-partite hypergraph \defn{complete} if every possible such edge is in the edge-set. A perfect matching of an $r$-uniform hypergraph $\mathcal{H}$ is a set of pairwise vertex-disjoint hyperedges containing every vertex of $\mathcal{H}$. 

We begin with a partitioning lemma, which generalises \cref{lem:vec-partition} to complete $r$-uniform hypergraphs.
\begin{lemma}\label{lem:hyper-partition}
    Let $\vecf\colon E(K_{rn}^{(r)})\to\R^{k}$ satisfy $\|\vecf(e)\|_1\leq 1$ for each edge $e$. Then there is a balanced complete $r$-partite subhypergraph $H$ satisfying $f_h(H)/(nr)^{r-1}=\bigO(\sqrt{k})$.
\end{lemma}

\begin{proof}
    As in the proof of \cref{lem:vec-partition}, by shifting and scaling $h$ we may assume that $\vecf(K_{rn}^{(r)})=0$ and $\norm{\vecf(e)}_1\leq 1$ for all $e \in E(K_{rn}^{(r)})$. Consider a random equipartition of $V(K_{rn}^{(r)})$ into $r$ vertex parts. For each edge $e \in E(K_{rn}^{(r)})$, define $X_e$ to be $h(e)$ if $e$ spans all $r$ vertex parts and $0$ otherwise and let $X=\sum_{e\in E}X_e$. Suppose $e$ and $e'$ are any two fixed edges in $E(K_{rn}^{(r)})$.
    Let ${p=n^r/\binom{rn}{r}}$ denote the probability that $e$ is spanning, and let ${p_s=(n-1)^{r-s}/\binom{(n-1)r}{r-s}}$ denote the probability that $e'$ is spanning given that $e$ is spanning and $|e \cap e'| = s$. 
    We will bound $\E\left(\|X\|_2^2\right)$. First note that we have
    \[
        \E X_e X_{e'} = \vecf(e)\vecf(e')\cdot p\cdot p_s.
    \]
     Letting $\sigma_s\coloneq\sum_{e,e'\colon |e\cap e'|=s}h(e) h(e')=\bigO(n^{2r-s})$, it follows that
     \begin{equation}\label{eq:hypergraph_expectation}
        \E\|X\|_2^2 = \sum_{e,e'}\E X_e X_{e'} = p\cdot\sum_{s=0}^r p_s\sigma_s = p\cdot\left(p_0\sigma_0+p_1\sigma_1+p_2\sigma_2+\bigO(n^{2r-3})\right).
     \end{equation}
     To complete the proof, we require more precise bounds for the values of the first three terms. As in the proof of \cref{lem:vec-partition}, we will write each term in terms of the `degrees' of subsets of the vertex set. For each set $S \subseteq V(K_{rn}^{(r)})$, we consider $\sum_{e \supseteq S}h(e)$ to be its `degree' in $K_{rn}^{(r)}$, analogous to the $d_v$ terms used in \cref{lem:vec-partition}. We introduce the sum of squares of these terms.
     \begin{align*}
         D_s \coloneq \sum_{S\colon|S|=s} \left(\sum_{e \supseteq S}h(e)\right)^2
         = \sum_{e,e'} {|e \cap e'| \choose s}h(e)h(e')
         = \sum_{j=0}^r \binom{j}{s}\sigma_j
         .
     \end{align*}
     Applying a standard binomial inversion identity to $D_s$, we obtain
     \[
        \sigma_s = \sum_{j=0}^r (-1)^{s+j}\binom{j}{s}D_j.
     \]
     Finally, observe that $p_1=(1-\frac{r-1}{(n-1)r})p_0$, and so $p_0\geq p_1$. Furthermore, we have that ${D_0=0}$, $D_1\geq 0$ and $D_2 \leq \binom{nr}{2}\binom{nr-2}{r-2}^2\leq\frac{1}{2}(nr)^{2r-2}$. Substituting these values into \eqref{eq:hypergraph_expectation}, we get
     \begin{align*}
         \E\|X\|_2^2  &= p\cdot\left(p_0 D_0+(-p_0+p_1)D_1 + (p_0-2p_1+p_2)D_2 \right)+\bigO(n^{2r-3})\\
         &\leq   (nr)^{2r-2} + \bigO(n^{2r-3}).
    \end{align*}
    Hence, there exists a partition such that $\|X\|_2=\bigO((nr)^{r-1})$. Letting $H$ be the complete $r$-partite hypergraph with this partition, we have $f_\vecf(H)=\|X\|_1\leq{\sqrt{k}\|X\|_2}={\bigO(\sqrt{k}\cdot (nr)^{r-1})}$.
\end{proof}

We again use \cref{thm:vec-bipartite} to embed our matching into this $r$-partite subgraph. This time, we proceed iteratively, applying \cref{thm:vec-bipartite} at each stage to embed the matching one vertex part at a time.
\hypergraph*
\begin{proof}
    Partition $V(K_{rn}^{(r)})$ into vertex parts $V_1,V_2,\dots,V_r$ each of size $n$ such that the subgraph $H$ of $K_{rn}^{(r)}$ whose edge-set is every edge of $K_{rn}^{(r)}$ containing exactly one vertex from each of the $r$ vertex parts satisfies $f_\vecf (H) = \bigO((rn)^{r-1}\sqrt{k})$. Such a subgraph $H$ exists by \cref{lem:hyper-partition}. Our proof will proceed by induction, applying \cref{thm:vec-bipartite} to complete the induction step, with the following hypothesis.
    \begin{claim}
        Let $H^{(r)}$ be an $rn$-vertex complete $r$-partite $r$-uniform hypergraph with vertex parts of size $n$, and let $\vecf \colon E(H^{(r)}) \to \R^k$ such that $\| \vecf(e) \|_1 \leq 1$ for all $e \in E(H^{(r)})$. Then there is a perfect matching $M$ of $H^{(r)}$ satisfying $$f_\vecf(M) \leq Crk^2$$
        for some absolute constant $C$.
    \end{claim}
    \begin{poc}
        The claim holds when $r = 2$ by \cref{thm:vec-bipartite}, so suppose $r > 2$ and that the claim holds for $r-1$. Let the vertex parts of $H^{(r)}$ be $U_1,\dots,U_r$, each of size $n$, and consider the complete $(r-1)$-partite, $(r-1)$-uniform hypergraph $H^{(r-1)}$ defined on vertex parts $U_1,\dots,U_{r-1}$. Define $\vecf_1 \colon E(H^{(r-1)}) \to \R^k$ as follows. For all $e'$ in $E(H^{(r-1)})$, we have
        \[\vecf_1(e') = \frac{1}{n}\sum_{e \in H^{(r)} : e' \subseteq e} \vecf(e).\]
        Applying the induction hypothesis, and the fact that $\vecf_1(H^{(r-1)}) = \frac{1}{n}\vecf(H^{(r)})$, we obtain a perfect matching $M_1$ of $H^{(r-1)}$ satisfying 
        \begin{equation}\label{eq:hypergraph_1}
            f_{\vecf_1}(M_1) = \left\|\vecf_1(M_1) - \frac{e(M_1)}{e(H^{(r-1)})}\frac{1}{n}\vecf(H^{(r)}) \right\|_1 \leq C(r-1)k^2.
        \end{equation}
        
        We now construct from $H^{(r)}$ an auxiliary bipartite graph $K_{n,n}$, whose vertex parts are $E(M_1)$ and $U_r$. Each edge $(e', u)$ in $K_{n,n}$ with $e' \in E(M_1)$ and $u \in U_r$ corresponds to a hyperedge $e' \cup u$ in $H^{(r)}$. We therefore define $\vecf_2((e',u)) = \vecf(e' \cup u)$ on each edge $(e',u)$. Note that by definition of $\vecf_1$, we have that ${\vecf_2(K_{n,n}) = n\cdot \vecf_1(M_1)}$. By applying \cref{thm:vec-bipartite} to $K_{n,n}$, we obtain a matching $M_2$ of $K_{n,n}$ satisfying 
        \begin{equation}\label{eq:hypergraph_2}
            f_{\vecf_2}(M_2) = \left\|\vecf_2(M_2) - \frac{1}{n} \vecf_2(K_{n,n})\right\|_1 = \|\vecf_2(M_2) - \vecf_1(M_1)\|_1 \leq Ck^2.
        \end{equation}
        Let $M$ be the perfect matching of $H^{(r)}$ consisting of every edge $e' \cup u$ for some $(e', u)$ in $M_2$. By definition of $\vecf_2$, we have $\vecf(M) = \vecf_2(M_2)$. Noting that $\frac{e(M)}{e(H^{(r)})} = \frac{1}{n}\frac{e(M_1)}{e(H^{(r-1)})}$, and combining inequalities \ref{eq:hypergraph_1} and \ref{eq:hypergraph_2}, we achieve $$f_\vecf(M) \leq Ck^2 + C(r-1)k^2 = Crk^2,$$ as desired. 
    \end{poc}
    The theorem now follows immediately by applying the claim to $H$.
\end{proof}

\subsection{Embedding other spanning subgraphs}

We now consider the problem of finding a representative copy of an arbitrary spanning subgraph with bounded maximum degree. Our strategy is again broadly the same -- we find a suitable partition of $V(K_n)$, and then embed our copy of $H$ into the resulting partite graph. However, both partitioning and embedding are more difficult in this setting.

We begin with our embedding method. Let $H$ be an $n$-vertex $r$-partite graph with vertex parts $U_1,\dots,U_r$, and let $G$ be an $n$-vertex graph with vertex partition $V_1,\dots,V_r$ such that $|U_i| = |V_i|$ for each $i \in [r]$. A \defn{partwise embedding} of $H$ into $G$ is an embedding of $H$ in $G$ such that for each $i$, the vertex part $U_i$ is mapped bijectively into $V_i$. We will always work with graphs with similarly labelled vertex parts, so even when multiple parts have the same size, there is no ambiguity about which vertex parts of $H$ will be embedded in which vertex parts of $G$. We show that if $H$ has maximum degree $\Delta$, we can find a partwise embedding of $H$ into a vector-labelled $r$-partite $G$ that is close to representative of the average partwise embedding of $H$ into $G$.
\begin{theorem}\label{thm:arbitrary-embedding}
    Let $G$ be the complete $n$-vertex $r$-partite graph with vertex parts $V_1, V_2, \dots, V_r$, and let $\vecf \colon E(G) \to \mathbb{R}^k$ be such that $\| \vecf(e) \|_1 \leq 1$ for all $e \in E(G)$. Let $H$ be an $n$-vertex $r$-partite graph with maximum degree $\Delta$ and vertex parts $U_1, U_2, \dots, U_r$ such that $|V_i| = |U_i|$ for all $i \in [r]$. Then there is a partwise embedding $H'$ of $H$ in $G$ such that
    \[
        \|\vecf(H') -\mu^{pw}_\vecf(H)\|_1=\bigO(\Delta\cdot r\cdot k^2),
    \] where $\mu^{pw}_\vecf(H)$ denotes the average value of $h(H)$ over all partwise embeddings of $H$ into $G$. 
\end{theorem}
\begin{proof}
    Analogously to the start of \cref{lem:vec-partition}, by shifting and scaling $h$ we may assume that $\mu^{pw}_\vecf(H) = 0$. We say an embedding $\varphi$ that maps a subset $A$ of $V(H)$ injectively into $V(G)$ is a \defn{partial partwise embedding} if, for each $i \in [r]$, we have $\varphi(A \cap U_i) \subseteq V_i$. Our strategy is to construct a partwise embedding $\varphi$ that maps $V(H)$ into $V(G)$ by constructing $\varphi$ iteratively, at each stage considering a partial partwise embedding ${\varphi_{i-1} : \bigcup_{j=1}^{i-1}U_j \to \bigcup_{j=1}^{i-1}V_j}$. 

    Define $\varphi_0$ to be an embedding of the empty vertex set into $V(G)$. An \defn{extension} $\varphi'$ of $\varphi_{i-1}$ is any partial partwise embedding of $A \subseteq V(H)$ into $V(G)$ such that $\bigcup_{j=1}^{i-1}U_j \subseteq A$ and $\varphi'$ restricts to $\varphi_{i-1}$ on $\bigcup_{j=1}^{i-1}U_j$. A \defn{complete extension} of $\varphi_{i-1}$ is an extension of $\varphi_{i-1}$ that bijectively maps the complete vertex set $V(H)$ into $V(G)$. We will construct the extension $\varphi_i$ of $\varphi_{i-1}$ by determining a bijection from $U_i$ into $V_i$ such that the remaining set of complete extensions of $\varphi_i$ is still close to representative of the original edge-set of $G$. For each partial partwise embedding $\varphi_i$, we denote by $\mu^{pw}_\vecf(H \mid \varphi_i)$ the average value of $h(\varphi_r(H))$ over every possible complete extension $\varphi_r$ of $\varphi_i$.
    
    For each $i \in [r]$, we construct a partial partwise embedding $\varphi_i\colon\bigcup_{j=1}^{i}U_j\to\bigcup_{j=1}^{i}V_j$ by extending $\varphi_{i-1}$ as follows. Let $B'=\bigcup_{j=1}^{i-1}V_j$ denote the set of vertices of $G$ contained in vertex parts in the image of $\varphi_{i-1}$. For a vertex $a \in U_i$ and an edge $e=(b,v)$ of $G$ with $b\in V_i$, let $p_{a,e}$ denote the probability that $e$ is contained in the image of a uniform random complete extension of $\varphi_{i-1}$ that embeds $a$ at $b$. Explicitly, 
    \begin{equation*}
        p_{a,e} = 
        \begin{cases}
            0 & v \in B' \text{ and } (a, \varphi_{i-1}^{-1}(v)) \notin E(H), \\
            1 & v \in B' \text{ and } (a, \varphi_{i-1}^{-1}(v)) \in E(H), \\
            \frac{|E(\{a\},U_j)|}{|V_j|} & v\in V_j \text{ with } j>i.
        \end{cases}
    \end{equation*}

    Consider the auxiliary complete bipartite graph $K^{(i)}$ with vertex parts $U_i$ and $V_i$. For each edge $q=(a,b)$ of $K^{(i)}$, define $\vecf_i\colon E(K^{(i)})\to\R^k$ by 
    \[
        \vecf_i(q) = \sum_{v \in V(G)}p_{a,bv}\cdot\vecf(bv).
    \]
    Observe that for any $q\in E(K^{(i)})$ we have $\|\vecf_i(q)\|_1\leq \sum_{v\in V(G)}p_{a,bv} = d_{H}(a) \leq \Delta$. Furthermore, we have $\frac{1}{|V_i|}\sum_{q\in E(K^{(i)})} \vecf_i(q)=\mu^{pw}_\vecf(E(U_i,V(H)\setminus U_i)\mid\varphi_{i-1})$. By applying \cref{thm:vec-bipartite} to $K^{(i)}$ and $\frac{1}{\Delta}\vecf_i$, we obtain a perfect matching $M$ of $K^{(i)}$ that satisfies 
    \[
        \| \vecf_i(M)-\mu^{pw}_\vecf(E(U_i,V(H)\setminus U_i)\mid\varphi_{i-1})\|_1=\bigO(\Delta\cdot k^2).
    \]
    Finally, we define $\varphi_i$ to be the extension of $\varphi_{i-1}$ where each vertex in $U_i$ is mapped to the unique vertex of $V_i$ that it is matched to by $M$. It follows that
    \[
        \|\mu_\vecf^{\text{pw}}(H \mid \varphi_i)-\mu_\vecf^{pw}(H \mid \varphi_{i-1})\|_1=\bigO(\Delta\cdot k^2),
    \]
    and so by induction, $\varphi_r$ is a valid complete embedding of $H$ such that ${\|\vecf(\varphi_r(H))\|_1=\bigO(\Delta \cdot r \cdot k^2)}$, which completes the proof.
\end{proof}

\cref{thm:arbitrary-embedding} states that we can find an embedding of any bounded-degree $r$-partite $H$ into a vector-labelled $r$-partite $G$ that is close to representative of the average partwise embedding. We will now establish a generalisation of \cref{lem:vec-partition}, which says that, under certain conditions on $H$, we can partition a vector-labelled $K_n$ into suitable vertex parts $V_i$ such that the average partwise embedding of $H$ is close to the expected weight $\frac{e(H)}{e(K_n)}h(K_n)$ of a random embedding of $H$ in $K_n$. The condition on $H$ is that the vertex degrees in $H$ are spread reasonably uniformly among the vertex parts. Formally, let $H$ be an $n$-vertex graph with average degree $d$. Let $U_1,\dots,U_r$ be a partition of $V(H)$ into $r$ independent sets, and for each $i \in [r]$, let $n_i = |U_i|$ and let $d_i=\frac{1}{n_i}\sum_{v\in U_i} d(v)$ denote the average degree of the vertices in $U_i$. We say $U_1,\dots,U_r$ is an \defn{$(r,C)$-partition} of $V(H)$ if $\sum_{i\in[r]}\frac{n_i}{n}(d_i-d)^2\leq C^2$. If $H$ admits an $(r,C)$-partition, we say it is \defn{$(r,C)$-uniform}. Observe that if $d_i\in[d-C,d+C]$ for each $i\in[r]$, then the partition $U_1,\dots,U_r$ is an $(r,C)$-partition of $V(H)$.

\begin{lemma}\label{lem:graph-uneven-partition}
    Let $H$ be an $n$-vertex, $(r,C)$-uniform graph with maximum degree $\Delta$ and $(r,C)$-partition $U_1,\dots,U_r$. Let ${\vecf\colon E(K_n)\to\R^k}$ be such that $\|\vecf(e)\|_1 \leq 1$ for all $e \in E(K_n)$, and let $\mu_\vecf(H)=\frac{e(H)}{e(K_n)}\vecf(K_n)$ denote the average value of $\vecf(H)$ over all embeddings of $H$ into $K_n$. 
    Then there is a partition $V(K_n)=\bigcup_{i=1}^r V_i$ such that $|V_i|=|U_i|$ and 
    \[
        \|\mu^{pw}_\vecf(H) - \mu_\vecf(H) \|_1 \leq \bigO(\sqrt{k(\Delta^2 r + C^2 n)}),
    \]
    where $\mu^{pw}_\vecf(H)$ denotes the average value of $\vecf(H)$ over all partwise embeddings of $H$ into $K_n$.
\end{lemma}

\begin{proof}
    As in \cref{lem:vec-partition}, by shifting and scaling $\vecf$ by a factor of $\frac{1}{2}$ we may assume
    that $\vecf(K_n)=0$ and $\|\vecf(e)\|_2\leq\|\vecf(e)\|_1\leq 1$ for all $e \in E(K_n)$. Under this assumption, we will find a partition that satisfies $\|\mu_\vecf^{pw}(H)\|_2^2\leq \bigO(\Delta^2r+C^2n)$. Since ${\|\mu_\vecf^{pw}(H)\|_1\leq\sqrt{k}\cdot\|\mu_\vecf^{pw}(H)\|_2}$, this suffices to complete the proof.

    Let $d$ be the average degree of $H$. For each $i \in [r]$, let $n_i=|U_i|$ and let $d_i$ denote the average degree of $U_i$. Finally, for each pair $i,j \in [r]$, not necessarily distinct, let $\rho_{ij}\coloneq\frac{e(U_i,U_j)}{n_i n_j}$ denote the edge-density of $H$ between $U_i$ and $U_j$. Note that $\rho_{ii} = 0$ for each $i \in [r]$. It will be useful for us to introduce several quantities that depend only on $H$ and its $(r,C)$-partition. Each of them in some sense measures the average value of squares of degrees in $H$, and in particular is bounded above by a multiple of $\Delta^2$. 
    Let 
    \[
        M=\sum_{i\in[r]}\frac{n_i}{n}d_i^2 = \sum_{i,j,k}\frac{n_i}{n}n_jn_k\rho_{ij}\rho_{ik} 
    \]
    be the weighted average of squares of $d_i$, where the second sum is over (not necessarily distinct) triples $i,j,k \in [r]$. Clearly $M\leq\Delta^2$, and by definition of $(r,C)$-uniformity we have that $R\coloneq M-d^2\leq C^2$. Additionally, let \[Q=\sum_{i,j}n_i n_j\rho_{ij}^2=\sum_{i,j}\frac{e(U_i,U_j)}{n_i}\frac{e(U_i,U_j)}{n_j},\] where both sums are over (not necessarily distinct) pairs $i,j \in [r]$. Then, letting $d_{ij}$ denote $\frac{e(U_i,U_j)}{n_i}$, we have that $\sum_{j}d_{ij}=d_i\leq\Delta$, and therefore 
    \[
        Q=\sum_{i,j}d_{ij}\cdot d_{ji}\leq\sum_{i,j}\Delta \cdot d_{ji}
        =\sum_i \Delta \cdot d_i \leq \sum_i \Delta^2 = r\Delta^2.
    \]
    
    Consider now a uniformly random partition $V(K_n)=\bigcup_{i=1}^r V_i$ such that $|V_i|=|U_i|$ for each $i \in [r]$. For each edge $e\in E(K_n)$ with endpoints in $V_i$ and $V_j$, let $X_e$ be equal to $\rho_{ij}$. Then we have
    \[
        X\coloneq\mu_\vecf^\text{pw}(H)=\sum_{e\in E(K_n)}\vecf(e)\cdot X_e.
    \]
    Observe that for $e',e''\in E(K_n)$, the expectation $\E X_{e'}X_{e''}$ depends only on the value of ${s\coloneq|e'\cap e''|\in\{0,1,2\}}$. We denote this expectation by $K_s$. Similarly to in the proof of \cref{lem:hyper-partition}, for each $s\in\{0,1,2\}$, we let $\sigma_s=\sum_{e,e'\colon|e\cap e'|=s}h(e)h(e')$ and ${D_s=\sum_{|S|=s}(\sum_{e\supseteq S}h(e))^2}$. It follows that
    $\sigma_0 = D_0 - D_1 + D_2$, $\sigma_1 = D_1 - 2D_2$, and $\sigma_2 = D_2$. Therefore,
    \begin{equation}\label{eq:EX2}
        \E \|X\|_2^2 = K_0\sigma_0+K_1\sigma_1+K_2\sigma_2=K_0D_0+(K_1-K_0)D_1+(K_0-2K_1+K_2)D_2.
    \end{equation}
    We now show that $K_0$, $K_1$ and $K_2$ differ by a factor of $n^{-2}(1+\bigO(n^{-1}))$ from $d^2$, $M$ and $Q$, respectively. The simplest case is $K_2$.
    \[
        K_2 = \E X_e^2 = \frac{1}{n(n-1)}\sum_{i,j\in[r]}n_i n_j\rho_{ij}^2=\frac{1}{n(n-1)}Q.
    \]
    To calculate $K_1$, we average $\rho_{ij}\rho_{ik}$ over ordered triplets $(v_1,v_2,v_3)$ of distinct vertices from (not necessarily distinct) vertex parts $U_i$, $U_j$ and $U_k$, respectively. In the following calculation, we sum over all ordered triplets (of not necessarily distinct vertices) and subtract the contribution of triplets where $v_1=v_3$. Since $\rho_{ii} = 0$ for all $i \in [r]$, this calculation gives precisely the sum over all ordered triplets of distinct vertices. Letting $(n)_k$ denote the falling factorial $n(n-1)\cdots(n-k+1)$, we have
    \[
        K_1= \frac{1}{(n)_3}\left(\sum_{i,j,k}n_i n_j n_k \rho_{ij}\rho_{ik}-\sum_{i,j}n_i n_j \rho_{ij}^2\right) = \frac{1}{(n)_3}(nM-Q).
    \]
    By similar calculations, averaging over quadruplets of distinct vertices that define two disjoint edges, we have
    \[
        K_0=\frac{1}{(n)_4}\left(n^2d^2-4nM + 2Q\right).
    \]
    Now, since $M\geq d^2$ and $Q\geq 0$, it follows that
    \begin{align*}
        K_0-2K_1+K_2 &\leq \frac{1}{(n)_4}(n^2d^2 - 2n(n-1)M + (n-1)(n-2)Q) \\
        &\leq \frac{1}{(n)_4}((n-1)(n-2)Q),
        \intertext{and}
        K_1-K_0      &\leq \frac{1}{(n)_4}\left(nM + n^2(M-d^2)\right).
    \end{align*}
    Finally, we substitute these values into \eqref{eq:EX2} with $D_0=0$, $D_1\leq n(n-1)^2$ and $D_2\leq\binom{n}{2}$. Recalling that $Q \leq \Delta^2r$, $M \leq \Delta^2$ and $M - d^2 \leq C^2$, then for $n \geq 4$ we have
    \[
        \E\|X\|_2^2\leq \frac{1}{(n)_4}\left(n(n-1)^2(nM+n^2C^2)+\frac{n(n-1)^2(n-2)}{2}r\Delta^2\right)\leq 6(\Delta^2 r + C^2 n),
    \]
    and so a partition with the desired property exists. \qedhere
\end{proof}

Combining the random partitioning approach with the embedding strategy now allows us to prove the following general result for $(r,C)$-uniform graphs
\begin{theorem}\label{thm:general_uniformity_error}
    Let $H$ be a $(r,C)$-uniform $n$-vertex graph with maximum degree $\Delta$, and let $\vecf : E(K_n) \to \mathbb{R}^k$ be such that $\|\vecf(e)\|_1 \leq 1$ for all $e \in E(K_n)$. Then there is a copy $H'$ of $H$ in $K_n$ satisfying
    \[f_\vecf(H') = \bigO(C\sqrt{kn} + \Delta r k^2).\]
\end{theorem}
\begin{proof}
    Let $U_1,\dots U_{r}$ be an $(r,C)$-partition of $V(H)$. Applying \cref{lem:graph-uneven-partition}, we obtain a partition $V_1,\dots,V_r$ of $V(K_n)$ such that $|V_i| = |U_i|$ for each $i \in [r]$, and 
    \begin{align*}
        \|\mu_\vecf^{pw}(H) - \mu_\vecf(H) \|_1 &\leq \bigO(\sqrt{k(\Delta^2 r + C^2 n)}) = \bigO(\Delta\sqrt{rk} + C\sqrt{kn}),
    \end{align*}
     where $\mu_\vecf^{pw}(H)$ denotes the average value of $\vecf(H)$ over all partwise embeddings of $H$ in $K_n$, and $\mu_\vecf(H) = \frac{e(H)}{e(K_n)}\vecf(K_n)$ is the average value of $h(H)$ over all embeddings of $H$ in $K_n$. By \cref{thm:arbitrary-embedding}, there is a partwise embedding $H'$ of $H$ in $K_n$ such that
    \begin{align*}
        \|h(H') - \mu_\vecf^{pw}(H) \|_1 = \bigO(\Delta r k^2).
    \end{align*}
    It follows that this copy $H'$ satisfies
    \begin{align*}
        f_\vecf(H') = \|h(H') - \mu_\vecf(H) \|_1 = \bigO(C\sqrt{kn} + \Delta r k^2),
    \end{align*}
    as required.
\end{proof}

To achieve an error-bound for an almost representative embedding of a subgraph of $K_n$, it now suffices to establish that the subgraph is sufficiently uniform. We begin with the special case of $F$-factors. 
 \begin{lemma}\label{lem:factor_uniformity}
     Let $F$ be a graph on $r$ vertices with maximum degree $\Delta$. The $n$-vertex $F$-factor $H$ is $(r,\frac{r^3}{n})$-uniform.
 \end{lemma}
 \begin{proof}
     Let $d$ denote the average degree of $F$ (and hence of $H$). We show that it is possible to partition $V(H)$ into $r$ independent sets $U_1,\dots,U_r$ each of size $\frac{n}{r}$ such that for each $U_i$, the sum of the degrees of vertices in $U_i$ is in $[\frac{n}{r}d-r^2, \frac{n}{r}d+r^2]$, from which the lemma statement follows.

     Fix a vertex ordering on $V(F)$ and consider the $r$ distinct cyclic permutations of this vertex ordering. Each permutation corresponds to a unique bijection between $V(F)$ and $\{U_1,\dots,U_r\}$. If $\frac{n}{r}$ = $jr$ for some positive integer $j$, then we may partition $V(H)$ so that each of these permutations appears as a copy of $F$ in $H$ precisely $j$ times. Since each vertex of $F$ is mapped into each vertex part of $H$ exactly $j$ times, it follows that the sum of vertex degrees in each $U_i$ is precisely $\frac{n}{r}d$. Hence, in general we consider a partition of $V(H)$ into parts such that each permutation appears either $j$ or $j+1$ times for some $j$. The maximum degree of $F$ is $r-1 < r$ and the minimum degree is $0$, so it follows that the degree sum of each $U_i$ is at least $(\frac{n}{r}-r)d$ and at most $(\frac{n}{r}-r)d + r^2$. Since $d < r$, this completes the proof.
 \end{proof}

Applying \cref{thm:general_uniformity_error} with the uniformity result from \cref{lem:factor_uniformity} proves the following.
 \begin{theorem}
    Let $F$ be an $r$-vertex graph with maximum degree $\Delta$, and let $H$ be an $n$-vertex $F$-factor. Let $\vecf : E(K_n) \to \mathbb{R}^k$ such that $\|\vecf(e)\|_1 \leq 1$ for all $e \in E(K_n)$. Then there is a copy $H'$ of $H$ in $K_n$ satisfying 
    \[f_\vecf(H') = \bigO(\Delta rk^2).\]
\end{theorem}

We now turn to the more general problem of arbitrary bounded-degree graphs. We will show that a bounded-degree graph $H$ admits a suitable partition by considering properties of a uniform random vertex colouring of $H$. The following concentration result of Chatterjee~\cite{chatterjee2005} will be useful.

\begin{lemma}[Proposition 4.5, \cite{chatterjee2005}] \label{lem:prop4.5}
    Let $H$ be an $n$-vertex graph with maximum degree $\Delta$ and let $X$ be a uniform random proper $q$-colouring of $H$. Let $g : [q]^{V(H)} \to \mathbb{R}$ satisfy ${|g(\sigma) - g(\sigma')| \leq \sum_{v \in V(H)} c_v \mathbb{1}_{\sigma(v) \neq \sigma'(v)}}$ for any two $q$-colourings $\sigma$ and $\sigma'$ of $H$. If $q > 2\Delta$, then for all $t \geq 0$, we have
    \begin{align*}
        \Prob[|g(X) - \E g(X)| \geq t] \leq 2\exp\left(-\frac{\gamma t^2}{\sum_{i=1}^n c_v^2}\right),
    \end{align*}
    where $\gamma = (q-2\Delta)/(q-\Delta)$.
\end{lemma}

We now establish our uniformity bound for arbitrary bounded degree subgraphs.
\begin{lemma}\label{lem:bounded_deg_uniformity}
    Let $H$ be an $n$-vertex graph with maximum degree $\Delta$. Then $H$ is $\left(3\Delta, \bigO\left(\Delta^2n^{-\frac{1}{2}}\sqrt{\log2\Delta}\right)\right)$-uniform.
\end{lemma}
\begin{proof}
    Let $H$ have average degree $d$. We first consider a proper vertex-colouring $\sigma$ of $H$ with $3\Delta$ colours. For any given $\sigma$, we let $U_i$ denote the set of vertices of $H$ coloured $i$. Define $n_i(\sigma) = |U_i|$ to be the number of vertices of $H$ coloured $i$ by $\sigma$, and $m_i(\sigma) = \sum_{v \in U_i}d(v)$ to be the sum of the degrees of colour $i$ vertices.
    
    Let $X$ be a uniform random proper $3\Delta$-colouring of $H$. Then we have that $\E[n_i(X)] = \frac{n}{3\Delta}$ and $\E[m_{i}(X)] = d\frac{n}{3\Delta}$. Furthermore, for each $i \in [3\Delta]$, and any distinct $(3\Delta)$-colourings $\sigma$ and $\sigma'$, we have that 
    \begin{align*}
        |n_i(\sigma) - n_i(\sigma')| &\leq \sum_{v \in V(H)}\mathbb{1}_{\sigma(v) \neq \sigma'(v)}, \\
        \intertext{and}
        |m_i(\sigma) - m_i(\sigma')| &\leq \sum_{v \in V(H)}d(v)\mathbb{1}_{\sigma(v) \neq \sigma'(v)}.
    \end{align*}
    Applying \cref{lem:prop4.5}, we find that $\gamma = 1/2$, and for each $i \in [3\Delta]$, we have
 \begin{align}
     \Prob \left( \left|n_i(X) - \frac{n}{3\Delta}\right| \geq t \right) &\leq 2\exp\left(-\frac{t^2}{2n}\right), \label{eq:vtx_concentration}
    \intertext{and}
     \Prob\left(\left|m_i(X) - d\frac{n}{3\Delta}\right| \geq t \right) &\leq 2\exp\left(-\frac{t^2}{2\sum_{v}d(v)^2}\right), \label{eq:degree_concentration}
 \end{align}
where $\sum_v d(v)^2 \leq n\Delta^2$. We have $3\Delta$ functions $n_i$ and $3\Delta$ functions $m_i$, and so we choose values for $t$ such that the probabilities in \eqref{eq:vtx_concentration} and \eqref{eq:degree_concentration} are both bounded above by $\frac{1}{7\Delta}$. Union-bounding over all $6\Delta$ events, we find that there exists a proper $3\Delta$ colouring $\sigma$ of $H$ such that ${\left|n_i(\sigma) - \frac{n}{3\Delta}\right| \leq \sqrt{2n\log(14\Delta)}}$ and $\left|m_i(\sigma) - d\frac{n}{3\Delta}\right| \leq \Delta\sqrt{2n\log(14\Delta)}$ for all $i \in [3\Delta]$. Let $U_1,\dots,U_{3\Delta}$ be the vertex partition of $H$ given by $\sigma$. Then the average degree of each $U_i$ is given by $\frac{m_i(\sigma)}{n_i(\sigma)}$, which satisfies
\[
    \frac{m_i(\sigma)}{n_i(\sigma)} - d 
    = \frac{(d+\Delta)\sqrt{2n\log(14\Delta)}}{\frac{n}{3\Delta}-\sqrt{2n\log(14\Delta)}}\leq \frac{2\Delta\sqrt{2n\log(14\Delta)}}{\frac{n}{6\Delta}}=\bigO(\Delta^2n^{-\frac{1}{2}}\sqrt{\log2\Delta}).
\]
The lower bound follows similarly, and so $H$ is $\left(3\Delta, \bigO\left(\Delta^2n^{-\frac{1}{2}}\sqrt{\log2\Delta}\right)\right)$-uniform, as desired. 
 \end{proof}

Applying \cref{thm:general_uniformity_error} with the uniformity result obtained in \cref{lem:bounded_deg_uniformity} completes the proof of our main result.
\coolbounds*

\section{Spanning forests}\label{sec:spanning_forests}

Embedding almost representative spanning forests in labelled complete graphs is a variant of our problem of particular interest, since such subgraphs have received considerable attention from the perspective of zero-sum Ramsey theory (see, for example, \cite{caro1996}, \cite{caro2022zero}, \cite{furedi1992zero}, \cite{mohr2022zero}). In particular, when $K_n$ has a zero-sum labelling by $\{-1,+1\}$ (equivalent to our $k=2$ colour case), the bound for embedding a fixed spanning forest $F$ is known to be $f_c(F) \leq \frac{\Delta}{2} + \bigO(1)$ \cite{hollom2024}. In this section, we prove that the error is linear in $\Delta$ for all values of $k$. We further show that, up to divisibility constraints, every colour-balanced complete graph admits some colour-balanced spanning tree, generalising a result of Caro, Hansberg, Lauri and Zarb \cite{caro2022zero} for the $2$-colour case.

\subsection{Embedding fixed spanning forests}

We begin with the proof of \cref{thm:forest_error}. The strategy is once again to apply \cref{thm:general_uniformity_error}, but with a stronger uniformity result that we will now obtain for spanning forests. Given a forest $F$ on $n$ vertices and a vertex $v \in V(F)$, we say that $v$ is a \defn{centroid} of $F$ if no component of $F \backslash v$ contains more than $\frac{n}{2}$ vertices. We say a vertex-colouring of a graph $G$ is \defn{$i$-dominant} if there are more than $\frac{|V(G)|}{2}$ vertices of $G$ coloured $i$.

\begin{lemma}\label{lem:forest_recurse}
    Let $F$ be a forest on $n$ vertices with maximum degree $\Delta$. For any integer $R \geq 1$ there is a set $X$ of at most $R$ vertices in $F$ such that $F \backslash X$ has a proper $2$-colouring in which each colour class has size in $[\frac{n-|X|}{2} - \frac{n}{2^R}, \frac{n-|X|}{2} + \frac{n}{2^R}]$. 
\end{lemma}
\begin{proof}
    Let $v_1$ be a centroid of $F$, and consider a proper $2$-colouring of $F \backslash v_1$. Suppose without loss of generality that this initial colouring is $2$-dominant. By interchanging the colour classes on each tree in $F \backslash v_1$ in turn, we obtain a $2$-dominant $2$-colouring $c_1$ of $F \backslash v_1$ such that for some component $T_1$ of $F \backslash v_1$, interchanging the colour classes in $T_1$ results in a $1$-dominant colouring of $F \backslash v_1$. Repeating this idea, we now find a centroid $v_2$ of $T_1$ and consider the effect on $c_1$ of interchanging the colour classes of each component of $T_1 \backslash v_2$ in turn. We obtain a $2$-dominant proper colouring $c_2$ of $F \backslash \{v_1, v_2\}$, and a component $T_2$ of $T_1 \backslash v_2$ such that interchanging the colour classes in $T_2$ results in a $1$-dominant colouring of $F \backslash \{v_1, v_2\}$. Observe that by the definition of a centroid, $T_1$ has at most $\frac{n}{2}$ vertices, and so $T_2$ has at most $\frac{n}{4}$ vertices. Iterating this process $m = \min(R, \lfloor\log_2(n)\rfloor)$ times, we obtain a set $X = \{v_1,\dots, v_m\}$ of at most $R$ vertices, and a $2$-dominant proper colouring $c_R$ of $F \backslash X$ such that interchanging the colour classes on a component of $F \backslash X$ with at most $\frac{n}{2^m}$ vertices produces a $1$-dominant colouring. The lemma statement immediately follows for $R \leq \log_2(n)$. If $R > \log_2(n)$, then we have obtained a proper $2$-colouring of $F \backslash X$ in which the sizes of the colour classes differ by at most $1$, and we can delete up to one additional vertex to precisely balance the colour classes. The lemma statement follows.
\end{proof}

We now obtain a bound on the uniformity of a spanning forest in terms of $\Delta$. Recall that an $n$-vertex graph $H$ with average degree $d$ is \defn{$(r,C)$-uniform} if there exists a partition of $V(H)$ into $r$ independent sets with sizes $\{n_i\}_{i\in[r]}$ and average degrees $\{d_i\}_{i\in[r]}$ such that $\sum_{i\in[r]}\frac{n_i}{n}(d_i-d)^2\leq C^2$.
\begin{lemma}\label{lem:forest_uniformity_fine}
    Let $F$ be a forest on $n$ vertices with maximum degree $\Delta$. Then $F$ is $(4,\bigO(\Delta n^{\frac{-1}{2}}))$-uniform.
\end{lemma}
\begin{proof}
    Let $c(F)$ denote the number of components of $F$, so that $F$ has $n-c(F)$ edges, and let $d = \frac{2(n-c(F))}{n}$ denote the average degree of $F$. We consider two cases.
    
    \noindent\textbf{Case 1: $\Delta \geq \frac{\sqrt n}{4}$.}
        Apply \cref{lem:forest_recurse} with $R=2$ to obtain a set $X$ of at most $2$ vertices such that $F \setminus X$ has a proper $2$-colouring with colour classes $V_1$ and $V_2$ each of $\Omega(n)$ size. Arbitrarily move up to $2$ vertices from $V_1$ into $X$ so that $|X| = 2$. Note that the total degree of each of $V_1$ and $V_2$ is at most $n-c(F) < n$, and so the average degree of each of $V_1$ and $V_2$ is $\bigO(1)$. Colour the vertices in $X$ with additional colours $3$ and $4$, respectively, to obtain a proper $4$-colouring of $F$. Each of the at most two additional colour classes contains at most $1$ vertex, and has average degree at most $\Delta$. Since $d<2$, it follows that
    \begin{align*}
        \sum_{i=1}^4 \frac{|V_i|}{n}(d_i-d)^2
        = \bigO\left(\frac{|V_1|+|V_2|}{n}\right)
          + \bigO\left(\frac{\Delta^2}{n}\right) = \bigO\left(\frac{\Delta^2}{n}\right),
    \end{align*}
    since $\Delta \geq \sqrt n/4$. So by definition, $F$ is $(4,O(\Delta n^{-1/2}))$-uniform.

    \noindent\textbf{Case 2: $\Delta \leq \frac{\sqrt n}{4}$.}
    Let $R = \lfloor\frac{\sqrt n}{\Delta}\rfloor$ and note that since $\Delta \leq \sqrt n/4$, we have $R > 1$. For simplicity, we will use $R = \frac{\sqrt{n}}{\Delta}$ in the following analysis, as it does not affect the asymptotics. Apply \cref{lem:forest_recurse} with $R$ to obtain a set $X$ of at most $R$ vertices such that $F\setminus X$ has a proper $2$-colouring with within $\frac{n}{2^R}$ of $\frac{n-|X|}{2}$ vertices in each of the two colour classes. Distribute the vertices of $X$ as evenly as possible between colours $1$ and $2$, and let $V_1$ and $V_2$ be the resulting (not necessarily proper) colour classes. Then $|V_i|$ is within $\frac{n}{2^R}$ of $\frac{n}{2}$ for each $i \in [2]$. Observe that since $|X| \leq R$, we have that $F[V_1] \cup F[V_2]$ is a forest containing at most $\Delta R = \sqrt{n}$ edges, and therefore at most $2\sqrt{n}$ vertices. We denote by $F_i$ the subgraph of non-isolated vertices in the induced forest $F[V_i]$.

    Our strategy is to partition each colour class into two classes to obtain a proper $4$-colouring of $F$ such that the colour classes have nearly equal sizes, and nearly equal total degrees. We will apply the same partitioning argument to both $V_1$ and $V_2$. Let $I$ denote the independent set of vertices in $V_1$ that are not incident with edges of $F_1$, and partition the vertices in $V_1$ into two sub-colour classes, $a$ and $b$, as follows. First, take any proper $2$-colouring of $F_1$ using $a$ and $b$. Then, for each possible degree $x$ of a vertex in $I$, proceed sequentially from $x=0$ to $x=\Delta$, and colour half of the degree-$x$ vertices in $I$ with $a$ and the other half with $b$, alternating which colour class receives an extra vertex when there are parity constraints. In this $2$-colouring of $I$, the vertices are divided as evenly as possible between $a$ and $b$, and the total degrees of each class $a$ and $b$ in $I$ are within $\Delta$ of each other. 
    
    Since $I$ contains at least $|V_1| - 2\sqrt{n}$ vertices, and the vertices in $I$ are divided evenly between $a$ and $b$, we have
    \begin{align*}
        \frac{\frac{n}{2} - \frac{n}{2^R} + \bigO(\sqrt{n})}{2} \leq |V_c| \leq \frac{\frac{n}{2} + \frac{n}{2^R} + \bigO(\sqrt{n})}{2}.
    \end{align*}
    Furthermore, since $I$ is incident with every edge in $F$ that is not incident with $F_1$ or $F_2$, the total degree of $I$ is at least $(n-c(F)) - 2\Delta\sqrt{n}$ and at most $(n-c(F))$. Each colour class $c \in \{a,b\}$ receives within $\Delta$ of the  $\frac{(n-c(F))}{2} + \bigO(\Delta\sqrt{n})$ equal share of the incident edges to $I$, and the additional contribution of the degrees from vertices in $F_1$ is $\bigO(\Delta\sqrt{n})$. Hence,
    \begin{align*}
        \sum_{v \in V_c} d(v) = \frac{(n-c(F)) +  \bigO(\Delta\sqrt{n})}{2}.
    \end{align*}
    Note that $\frac{n}{2^R} \leq \frac{n}{R} = \Delta\sqrt{n} \leq \frac{n}{4}$ and $d \leq 2$. It follows that each colour class $c$ has average degree $d_c$ satisfying
    \begin{align*}
        d_c \leq \frac{(n-c(F)) +  \bigO(\Delta\sqrt{n})}{\frac{n}{2} - \frac{n}{2^R} + \bigO(\sqrt{n})} = d + \frac{d\bigO(\frac{n}{2^R}+\sqrt{n}) + \bigO(\Delta\sqrt{n})}{\frac{n}{2} - \frac{n}{2^R} + \bigO(\sqrt{n})} = d + \frac{\bigO(\Delta\sqrt{n})}{\Omega(n)} = d + \bigO(\Delta n^{\frac{-1}{2}})
    \end{align*}
    and similarly $d_c \geq d - \bigO(\Delta n^{\frac{-1}{2}})$. Applying the same partitioning strategy with $2$ additional colours to $V_2$, we obtain a proper $4$-colouring of $F$ in which every colour class has an average degree within $\bigO(\Delta n^{\frac{-1}{2}})$ of the average degree $d$ of $F$. This completes the proof.
\end{proof}

Finally, we apply \cref{thm:general_uniformity_error} with our improved uniformity result for forests.
\forestError*
\begin{proof}
     By  
     \cref{lem:forest_uniformity_fine}, we have that $F$ is $(4,\bigO(\frac{\Delta}{\sqrt{n}}))$-uniform. Thus, applying \cref{thm:general_uniformity_error}, we have
     \begin{align*}
        f_\vecf(F') &= \bigO(\sqrt{8k(4\Delta^2 + \Delta^2)} + 4\Delta k^2) = \bigO(\Delta k^2),
    \end{align*}
     as required.
\end{proof}

\subsection{Existence of colour-balanced spanning trees}

We conclude with the proof of \cref{thm:colour_balanced_tree}, which characterises when edge-coloured complete graphs admit colour-balanced spanning trees, if we no longer fix the isomorphism class of the tree. To avoid parity constraints, we restrict our attention to $k$-edge-coloured graphs of the form $K_{2kt+1}$ for some integer $t \geq 1$. It is shown in \cite{caro2022zero} that when $k=2$ and $K_{4t+1}$ is coloured such that $|c^{-1}(i) \cap K_{4t+1}| > {2t \choose 2}$ for $i \in \{1,2\}$, there is always a colour-balanced spanning tree in $K_{4t+1}$. By applying a standard optimisation result for matroids, we show that this is just a special case of a more general condition for $k$-edge-coloured complete graphs. We begin with some necessary definitions.

A \defn{matroid} $M$ is an ordered pair $(E, \mathcal{I})$ consisting of a \defn{ground set} $E$ and a collection $\mathcal{I}$ of subsets of $E$ satisfying the following properties.
\begin{enumerate}
    \item[(I1)] $\varnothing \in \mathcal{I}$,
    \item[(I2)] If $I \in \mathcal{I}$ and $I' \subseteq I$ then $I' \in \mathcal{I}$, and
    \item[(I3)] If $I_1$ and $I_2$ are in $\mathcal{I}$ and $|I_1| < |I_2|$ then there is some element $e \in I_2 \backslash I_1$ such that $I_1 \cup \{e\} \in \mathcal{I}$.
\end{enumerate}
The sets in $\mathcal{I}$ are known as the \defn{independent sets} of $M$. Given a set $A \subseteq E$, the \defn{rank} of $A$ is the cardinality of the largest independent set contained in $A$, and the \defn{rank function} $r$ of $M$ is the function that maps subsets of $E$ to their rank. For further background on matroid theory, see \cite{oxley}. We will apply the following well-known result of Edmonds.
\begin{theorem}[Edmonds, \cite{edmonds1970submodular}]\label{thm:edmonds}
    Let $M_1$ and $M_2$ be matroids with independent sets $\mathcal{I}_1$ and $\mathcal{I}_2$, rank functions $r_1$ and $r_2$ and a common ground set $E$. Then
    \begin{align*}
        \max_{I \in \mathcal{I}_1 \cap \mathcal{I}_2}|I| = \min_{U \subseteq E}(r_1(U) + r_2(E\backslash U)).
    \end{align*}
\end{theorem}

\colourBalancedTree*
\begin{proof}
    We define two matroids on the common ground set $E = E(K_{2kt+1})$. Let ${M_1 = M(K_{2kt+1})}$ be the cycle matroid of $K_{2kt+1}$, that is the matroid whose independent sets $\mathcal{I}_1$ consist of the subsets of $E$ that induce forests in $K_{2kt+1}$. Let $M_2$ be the partition matroid on the collection of $k$ disjoint monochromatic subsets of $E$, each with capacity $2t$. That is, the independent sets $\mathcal{I}_2$ of $M_2$ are any subset of $E$ containing at most $2t$ edges from each colour class. Let $r_1$ and $r_2$ denote the respective rank functions of $M_1$ and $M_2$. Observe that if $I \in \mathcal{I}_1 \cap \mathcal{I}_2$, then $I$ is a forest of $K_{2kt+1}$ containing at most $2t$ edges in each colour class. If $|I| = 2kt$, then $I$ is a colour-balanced spanning tree of $K_{2kt+1}$. Hence, by \cref{thm:edmonds}, it suffices to show that 
    \begin{equation*}
        \min_{U \subseteq E}(r_1(U) + r_2(E \backslash U)) = 2kt.
    \end{equation*}
    Suppose then that $U$ is a minimiser of this function, and consider $r_2(E \backslash U)$. If $E \backslash U$ contains at most $2t$ edges in some colour $i$, and $e$ is a colour $i$ edge contained in $E \backslash U$, then we have that $r_2(E \backslash U - e) = r_2(E \backslash U) - 1$ and $r_1(U \cup e) \leq r_1(U) + 1$. It follows that moving $e$ into $U$ does not increase the value of $r_1(U) + r_2(E \backslash U)$, and so we may assume that $E \backslash U$ contains more than $2t$ edges from each colour class it intersects. Let $j$ denote the number of colour classes that are contained entirely in $U$. Since $E \backslash U$ contains more than $2t$ edges from each of the $k-j$ remaining classes, it follows that $r_2(E \backslash U) = (k-j)2t$. Moreover, $U$ contains at least ${2jt \choose 2} + 1$ edges by assumption, and $r_1(U)$ is given by the number of edges in a spanning forest of the edge-induced subgraph $K_{2kt+1}[U]$. The smallest clique large enough to contain every edge in $U$ has $2jt + 1$ vertices, and so $r_1(U) \geq 2jt$. Hence,
    \begin{align*}
        r_1(U) + r_2(E \backslash U) \geq 2jt + (k-j)2t \geq 2kt,
    \end{align*}
    as required.
\end{proof}
\cref{cor:balanced_tree} now follows immediately from \cref{thm:colour_balanced_tree}. We remark that the bounds in \cref{thm:colour_balanced_tree} are sharp. Indeed, the proof gives a construction that verifies this -- any edge-colouring of $K_{2kt+1}$ in which all edges labelled by some set of $j$ colours are contained in a clique on $2jt$ vertices does not admit a colour-balanced spanning tree. Hence, if there are at most ${2jt \choose 2}$ such edges, a colour-balanced spanning tree cannot be guaranteed.
 
\section{Lower bounds}\label{sec:lower-bounds}

In this section we establish lower bounds for the colouring problem in both the bipartite and complete graph settings (which also establishes the same lower bounds for the vector form of each problem). In the complete bipartite graph, we show that for all values of $t$ and $k$, it is not possible to guarantee the existence of a colour-balanced perfect matching in $K_{kt,kt}$, and so there must always be at least some constant error term. Furthermore, in \cref{thm:lower-Knn} we show that in many cases, every perfect matching $M$ of $K_{kt,kt}$ satisfies $f_c(M) \geq \sqrt{k/2}$.

In the complete graph $K_{2kt}$, it appears the only previously known lower bound comes from the $3$-edge-coloured $K_6$ counterexample to the existence of colour-balanced perfect matchings identified by Pardey and Rautenbach \cite{pardey2022}. It is natural to ask whether colour-balanced perfect matchings of colour-balanced complete graphs may always exist, provided $t$ is sufficiently large with respect to $k$. In \cref{thm:lower_Kn} we extend the counterexample of Pardey and Rautenbach to an infinite family of counterexamples, proving that this is not the case. In fact, the existence of a colour-balanced perfect matching in a colour-balanced $K_{2kt}$ is only ever guaranteed in the special case where $k=2$.

For brevity in this section, we will refer to edge-colourings that do not admit colour-balanced perfect matchings as being \defn{cbm-avoiding}. We begin with bounds for the complete bipartite graph. The following argument uses a slight modification of a construction for equi-$n$-squares with no transversals of size $n-\bigO(\sqrt{n})$ given by Chakraborti, Christoph, Hunter, Montgomery, and Petrov~\cite{chakraborti2024}.
\begin{theorem}\label{thm:lower-Knn}
    For infinitely many pairs $(k,t)$, there exist colour-balanced $k$-edge-colourings of $K_{kt,kt}$ such that every perfect matching $M$ of $K_{kt,kt}$ satisfies $f_c(M) \geq \sqrt{k/2}$.
\end{theorem}
\begin{proof}
    Let $m$ and $t$ be positive integers with $t$ odd and set $k=2m^2$. Let the two vertex parts of $K_{kt,kt}$ be $A$ and $B$, and partition each part into $2m$ subsets, labelled such that $A=\bigcup_{i\in[2m]} A_i$ and $B=\bigcup_{i\in[2m]} B_i$, with $|A_i|=|B_i|=mt$ for all $i\in[2m]$. For each pair $\{i,j\}\subseteq [2m]$ with $i<j$, introduce a colour $c_{ij}$ and colour every edge in $E(A_i,B_j) \cup E(A_j,B_i)$ with $c_{ij}$. For each $i\in[m]$, introduce a colour $c_{ii}$ and colour every edge in $E(A_{2i-1},B_{2i-1}) \cup E(A_{2i},B_{2i})$ with $c_{ii}$. The adjacency matrix representation of the colouring we obtain is shown on the left of~\cref{fig:bipartite_constructions}. Observe that we have used $\frac{(2m)^2}{2}=k$ distinct colours, and each colour class induces two complete bipartite copies of $K_{mt,mt}$. Fix a perfect matching $M$. We call a colour \defn{rare} if $M$ has strictly fewer than $t$ edges of that colour. We will show that there are at least $\frac{m}{2}$ rare colours, implying that $f_c(M) \geq m = \sqrt{k/2}$, from which the theorem follows. For each $i \in [m]$, let $S_i$ be the set of colours that appear on edges incident with $A_{2i-1} \cup A_{2i}$. Each colour is incident with at most two of the sets $A_j$, and hence belongs to at most two of the sets $S_i$. It therefore suffices to show that each $S_i$ contains a rare colour. 
    
    Fix $i\in [m]$ and suppose that no colour in $S_i$ is rare. Since $c_{ii}\in S_i$, it follows that $M$ has at least $t$ edges in $E(A_{2i-1},B_{2i-1}) \cup E(A_{2i},B_{2i})$, and so either $M\cap E(A_{2i-1},B_{2i-1})$ or $M\cap E(A_{2i},B_{2i})$ has size at least $\frac{t+1}{2}$, since $t$ is odd. Without loss of generality, assume that $|M \cap E(A_{2i},B_{2i})| \geq \frac{t+1}{2}$. Let $E_i$ be the set of edges incident with $A_{2i} \cup B_{2i}$ which do not belong to $E(A_{2i},B_{2i})$. Since each vertex of $A_{2i} \cup B_{2i}$ is matched exactly once by $M$, we have that
    \[
        |M \cap E_i| = |A_{2i}|+|B_{2i}|-2|M \cap E(A_{2i},B_{2i})| < mt + mt-t = (2m-1)t.
    \]
     Now, each $E_i$ contains $2m-1$ colour classes (indexed by pairs $\{2i,j\}$ with $j\neq 2i$) and thus, by the pigeonhole principle, one of those colours appears on fewer than $t$ edges in $M$. This contradicts the assumption that $S_i$ does not contain a rare colour, completing the proof. 
\end{proof}

The construction given in the proof of \cref{thm:lower-Knn} requires $k$ to be equal to $2m^2$ for some integer $m$. We also show that there exist colour-balanced $k$-edge-colourings of $K_{kt,kt}$ that do not admit colour-balanced perfect matchings, for all values of $k \geq 2$ and $t \geq 1$.

\begin{theorem}\label{thm:knn_lower2}
    For any $k\geq2$ and $t\geq 1$, there exists a colour-balanced $k$-edge-colouring of $K_{kt,kt}$ with no colour-balanced perfect matching.
\end{theorem}

\begin{proof}
    Let the two vertex parts of $K_{kt,kt}$ be $A$ and $B$. Partition each part into $k$ subsets labelled such that $A=\bigcup_{i\in [k]} A_i$ and $B=\bigcup_{i\in [k]} B_i$, with $|A_i| = t$ for all $i$, $|B_i| = t$ for all $1 < i < k$, $|B_1| = t-1$ and $|B_k| = t+1$. Consider the colour-balanced $k$-edge-colouring $c : E(A,B) \to [k]$ where, for each pair $(i,j)$, the edges in $E(A_i, B_j)$ are assigned the colour $i + j \pmod{k}$. The adjacency matrix representation of this colouring is shown on the right in~\cref{fig:bipartite_constructions}. Suppose that there exists a colour-balanced perfect matching $M$ of $K_{kt,kt}$. Since every colour occurs on exactly $t$ edges in $M$, we have that ${\sum_{e\in M}c(e)=t\sum_{i=1}^k i}$, which is equal to either $0 \pmod k$ or possibly $\frac{k}{2}$ if $k$ is even. However, since every vertex of $K_{kt,kt}$ is matched exactly once by $M$, we have
    \[
        \sum_{e\in M}c(e) = \sum_{i=1}^k (|A_i|+|B_i|)\cdot i = \left( 2t\sum_{i=1}^k i\right)+k-1 \equiv -1\pmod{k}
    \]
    which is not equal to either $0$ or $\frac{k}{2}$ when $k\geq 3$. Similarly, if $k=2$ and $t$ is even, we have $t\sum_{i=1}^k i\equiv 0\not\equiv -1\pmod{k}$. Hence, in both cases, this is a cbm-avoiding colouring of $K_{kt,kt}$.

    Finally, when $k=2$ and $t$ is odd, we instead take $|A_1| = |A_2| = |B_1| = |B_2| = t$, but keep the same edge-colouring rule as before. Now $\sum_{e\in M}c(e)$ is even for any perfect matching $M$ of $K_{2t,2t}$, while for a colour-balanced perfect matching the sum is $t+2t \equiv 1 \pmod 2$. Hence, we again have a cbm-avoiding colouring of $K_{2t, 2t}$, completing the proof. 
\end{proof}

\begin{figure}[h]
    \centering
    \begin{tikzpicture}[x=1cm,y=1cm]

  \newcommand{\celllabel}[2]{%
    \ifnum#1=1
      \ifnum#2=1 $c_{14}$\fi
      \ifnum#2=2 $c_{24}$\fi
      \ifnum#2=3 $c_{34}$\fi
      \ifnum#2=4 $c_{22}$\fi
    \fi
    \ifnum#1=2
      \ifnum#2=1 $c_{13}$\fi
      \ifnum#2=2 $c_{23}$\fi
      \ifnum#2=3 $c_{22}$\fi
      \ifnum#2=4 $c_{34}$\fi
    \fi
    \ifnum#1=3
      \ifnum#2=1 $c_{12}$\fi
      \ifnum#2=2 $c_{11}$\fi
      \ifnum#2=3 $c_{23}$\fi
      \ifnum#2=4 $c_{24}$\fi
    \fi
    \ifnum#1=4
      \ifnum#2=1 $c_{11}$\fi
      \ifnum#2=2 $c_{12}$\fi
      \ifnum#2=3 $c_{13}$\fi
      \ifnum#2=4 $c_{14}$\fi
    \fi
  }

  \newcommand{\cellcolor}[2]{%
    \ifnum#1=1
      \ifnum#2=1 myblue!35\fi
      \ifnum#2=2 myred!15\fi
      \ifnum#2=3 myyellow!15\fi
      \ifnum#2=4 mygreen!15\fi
    \fi
    \ifnum#1=2
      \ifnum#2=1 myred!40\fi
      \ifnum#2=2 myblue!15\fi
      \ifnum#2=3 mygreen!15\fi
      \ifnum#2=4 myyellow!15\fi
    \fi
    \ifnum#1=3
      \ifnum#2=1 myyellow!40\fi
      \ifnum#2=2 mygreen!35\fi
      \ifnum#2=3 myblue!15\fi
      \ifnum#2=4 myred!15\fi
    \fi
    \ifnum#1=4
      \ifnum#2=1 mygreen!35\fi
      \ifnum#2=2 myyellow!40\fi
      \ifnum#2=3 myred!40\fi
      \ifnum#2=4 myblue!35\fi
    \fi
  }

  \foreach \r in {1,2,3,4} {
    \foreach \c in {1,2,3,4} {
      \fill[\cellcolor{\r}{\c}] (\c-1,\r-1) rectangle (\c,\r);
    }
  }

  \draw (0,0) grid (4,4);

  \foreach \r in {1,2,3,4} {
    \foreach \c in {1,2,3,4} {
      \node at (\c-0.5,\r-0.5) {\celllabel{\r}{\c}};
    }
  }

  \foreach \c in {1,2,3,4} {
    \node[above] at (\c-0.5,4) {$A_{\c}$};
  }
  \foreach \r in {1,2,3,4} {
    \pgfmathtruncatemacro{\Bindex}{5-\r} %
    \node[left] at (0,\r-0.5) {$B_{\Bindex}$};
  }

\end{tikzpicture}
    \hspace{2cm}
    \begin{tikzpicture}[x=1cm,y=1cm]

  \newcommand{\celllabel}[2]{%
    \ifnum#1=1
      \ifnum#2=1 $1$\fi
      \ifnum#2=2 $2$\fi
      \ifnum#2=3 $3$\fi
      \ifnum#2=4 $4$\fi
    \fi
    \ifnum#1=2
      \ifnum#2=1 $4$\fi
      \ifnum#2=2 $1$\fi
      \ifnum#2=3 $2$\fi
      \ifnum#2=4 $3$\fi
    \fi
    \ifnum#1=3
      \ifnum#2=1 $3$\fi
      \ifnum#2=2 $4$\fi
      \ifnum#2=3 $1$\fi
      \ifnum#2=4 $2$\fi
    \fi
    \ifnum#1=4
      \ifnum#2=1 $2$\fi
      \ifnum#2=2 $3$\fi
      \ifnum#2=3 $4$\fi
      \ifnum#2=4 $1$\fi
    \fi
  }

  \newcommand{\cellcolor}[2]{%
    \ifnum#1=1
      \ifnum#2=1 myblue!25\fi
      \ifnum#2=2 mygreen!25\fi
      \ifnum#2=3 myyellow!25\fi
      \ifnum#2=4 myred!25\fi
    \fi
    \ifnum#1=2
      \ifnum#2=1 myred!25\fi
      \ifnum#2=2 myblue!25\fi
      \ifnum#2=3 mygreen!25\fi
      \ifnum#2=4 myyellow!25\fi
    \fi
    \ifnum#1=3
      \ifnum#2=1 myyellow!25\fi
      \ifnum#2=2 myred!25\fi
      \ifnum#2=3 myblue!25\fi
      \ifnum#2=4 mygreen!25\fi
    \fi
    \ifnum#1=4
      \ifnum#2=1 mygreen!25\fi
      \ifnum#2=2 myyellow!25\fi
      \ifnum#2=3 myred!25\fi
      \ifnum#2=4 myblue!25\fi
    \fi
  }

      \def\d{0.3}
    
      \def\yA{0}
      \def\yB{1+\d}   
      \def\yC{2+\d}
      \def\yD{3+\d}
      \def\yE{4}      


    \foreach \r/\yb/\yt in {1/\yA/\yB, 2/\yB/\yC, 3/\yC/\yD, 4/\yD/\yE}{
        \foreach \c in {1,2,3,4}{
            \fill[\cellcolor{\r}{\c}] (\c-1,\yb) rectangle (\c,\yt);
            \node at (\c-0.5, {\yb+0.5*(\yt-(\yb))}) {\celllabel{\r}{\c}};
        }
    }
    
  \foreach \x in {0,1,2,3,4}
    \draw (\x,\yA) -- (\x,\yE);

  \foreach \y in {\yA,\yB,\yC,\yD,\yE}
    \draw (0,\y) -- (4,\y);

  \foreach \c in {1,2,3,4} {
    \node[above] at (\c-0.5,4) {$A_{\c}$};
  }
  \foreach \r in {1,2,3,4} {
    \pgfmathtruncatemacro{\Bindex}{5-\r} %
    \node[left] at (0,\r-0.5) {$B_{\Bindex}$};
  }

\end{tikzpicture}
    \caption{Adjacency matrix representations of the constructions used in \cref{thm:lower-Knn} (left) with $k=8$, and \cref{thm:knn_lower2} (right) with $k=4$.}
    \label{fig:bipartite_constructions}
\end{figure}

We finish with our lower bound result for the complete graph. The following statement has already been demonstrated for $k=3$ and $t=1$ by Pardey and Rautenbach \cite{pardey2022}, via the $3$-edge-coloured $K_6$ shown in \autoref{fig:kn_counterexample_odd} (left).
\begin{theorem}\label{thm:lower_Kn}
    For any $k\geq 3$ and any $t\geq 1$, there exists a colour-balanced $k$-edge-colouring of $K_{2kt}$ with no colour-balanced perfect matching. 
\end{theorem}
\begin{proof}
    We prove the theorem in two parts, first for odd values of $k \geq 3$ and then for even values of $k \geq 3$. In both cases we will construct a colour-balanced cbm-avoiding colouring of $K_{2kt}$ by first defining an auxiliary cbm-avoiding colouring on a certain blow-up of $K_k$, and then making a small local modification to this colouring to make it colour-balanced. Examples of the resulting colourings are given in \cref{fig:kn_counterexample_odd} for odd $k$ and \cref{fig:kn_counterexample_even} for even $k$. We begin by defining some notation that will be useful in each of these constructions.

    The auxiliary colouring will always have the following structure. Partition the vertex set of $K_{2kt}$ into parts labelled $V_1,\dots,V_k$. For any integers $\Delta_1,\dots,\Delta_k$ satisfying $\sum_{i \in [k]} \Delta_i = 0$, let $c = c(\Delta_1,\dots,\Delta_k)$ be the edge-colouring of $K_{2kt}$ from palette $[k]$ where we let $|V_i| = 2t + \Delta_i$ for each $i \in [k]$, and for every (not necessarily distinct) $i$ and $j$, we colour every edge $e \in E(V_i, V_j)$ by $i + j \pmod k$. Since a perfect matching $M$ matches every vertex of $K_{2kt}$ once, it follows that
    \begin{equation}\label{eq:matching-mod-invariant}
        \sum_{e\in M}c(e) = 2t\sum_{i=1}^k i + \sum_{i=1}^k i \cdot \Delta_i \equiv \sum_{i=1}^{k}i\cdot \Delta_i \pmod{k}.
    \end{equation}
    However, in a colour-balanced perfect matching $M$, each colour appears exactly $t$ times, and so $\sum_{e\in M}c(e)=t\sum_{i=1}^k i$, which is either $0$ or $k/2 \pmod k$, depending on the parity of $k$ and $t$. Finally, let $m = \frac{1}{k}{2kt \choose 2}$ denote the size of each colour class in a colour-balanced colouring of $K_{2kt}$. 
    \begin{claim}\label{claim:counter-k-odd} 
        Let $k \geq 3$ be odd. For all $t \geq 1$ there exists a colour-balanced $k$-edge-colouring of $K_{2kt}$ with no colour-balanced perfect matching.
    \end{claim}
    \begin{poc}
         Consider the auxiliary colouring $c = c(\Delta_1,\dots,\Delta_k)$ with $\Delta_1 = -1$, $\Delta_k = 1$ and $\Delta_i = 0$ for all $1 < i < k$. It is straightforward to verify that  $c$ assigns $m$ edges to every colour class except for colours $1$ and $2$, which contain $m-1$ and $m+1$ edges respectively. Construct a colour-balanced $k$-edge-colouring of $K_{2kt}$ by changing the colour of a single edge $e$ in $E(V_2, V_k)$ from $2$ to $1$. Suppose for a contradiction that this colouring of $K_{2kt}$ admits a colour-balanced perfect matching $M$, and let $\delta\in\{0,1\}$ be the indicator of whether $e \in M$. By recolouring $e$ back from $1$ to $2$, we see that in the auxiliary colouring, $M$ satisfied $\sum_{e\in M}c(e)=\delta+t\sum_{i=1}^k i\equiv \delta \pmod{k}$. However, \eqref{eq:matching-mod-invariant} implies that this sum should be congruent to $-1 \pmod{k}$, so this is a contradiction.
    \end{poc}
    \begin{claim}\label{claim:counter-k-even} 
        Let $k \geq 3$ be even. For all $t \geq 1$ there exists a colour-balanced $k$-edge-colouring of $K_{2kt}$ with no colour-balanced perfect matching.
    \end{claim}
    \begin{poc}
        We start with the case where $t$ is odd. Consider the auxiliary colouring ${c_1 = c(\Delta_1,\dots,\Delta_k)}$ with $\Delta_i = 0$ for all $i \in [k]$. Noting that $k$ is even, it is straightforward to verify that $c_1$ colours $m+t$ edges with every odd colour, and $m-t$ edges with every even colour.
        Now fix a vertex $v\in V_k$, and for each odd $i\in [k]$, recolour $t$ edges in $E(\{v\},V_i)$ from $i$ to $i+1$, constructing a colour-balanced $k$-edge-colouring of $K_{2kt}$. Suppose for contradiction that this edge-colouring of $K_{2kt}$ admits a colour-balanced perfect matching $M$, and let $\delta\in\{0,1\}$ be the indicator of whether $v$ is matched by one of the recoloured edges. Then in $c_1$, $M$ satisfied $\sum_{e\in M}c_1(e)=-\delta+t\sum_{i=1}^k i \equiv \frac{k}{2}-\delta\not\equiv 0\pmod{k}$, contradicting \eqref{eq:matching-mod-invariant}. 

        When $t$ is even, consider the auxiliary colouring $c_2=c(\Delta_1, \dots, \Delta_k)$ with $\Delta_i = 0$ for all ${i \in [k] \setminus \{2,k\}}$ and $\Delta_2 = 1$, $\Delta_k = -1$. Observe that $c_2$ colours $m+t$ edges with every odd colour, and $m-t$ edges with every even colour except for colours $2$ and $k$, which label $m-t-1$ and $m-t+1$ edges, respectively. For each odd $i<k-1$, recolour $t$ edges in $E(\{v\}, V_{i})$ from $i$ to $i+1$, recolour $t-1$ edges in $E(\{v\}, V_{k-1})$ from $k-1$ to $k$ and recolour one additional edge in $E(\{v\}, V_{k-1})$ from $k-1$ to $2$. Since at most one recoloured edge is contained in a colour-balanced matching $M$, we have $\sum_{e\in M}c_2(e) \in \{0,-1,-3\} \pmod{k}$, and none of these residues is congruent to $2\pmod{k}$, contradicting \eqref{eq:matching-mod-invariant}.
    \end{poc}\qedhere
\end{proof}

\begin{figure}[h]
    \centering
    \raisebox{0.25cm}{\begin{tikzpicture}
    \def\c{2}
    \def\r{0.6}

    \begin{scope}[xshift=-0.866*\c cm, yshift=0.5*\c cm]
        \node[mycircle] (1a) at (0,0){};
        \node at (-0.5*\r,0.5*\r) {$V_1$};
    \end{scope}
    
    \begin{scope}[xshift=0.866*\c cm, yshift=0.5*\c cm]
        \foreach \x/\y in {0/1b,180/2b}{
            \node[mycircle] (\y) at (canvas polar cs: radius=\r cm,angle=\x-45){};
            \node at (0.5*\r,0.5*\r) {$V_2$};
        }
    \end{scope}
    
    \begin{scope}[yshift=-\c cm]
        \foreach \x/\y in {0/1c,120/2c,240/3c}{
            \node[mycircle] (\y) at (canvas polar cs: radius=\r cm,angle=\x+60){};
            \node at (0,-1.2*\r) {$V_3$};
        }
    \end{scope}

    \foreach \x in {1}{
        \foreach \y in {1,2}{
            \draw (\x a) edge[color=mygreen] (\y b);
        }
    }

    \foreach \x in {1,2}{
        \foreach \y in {1,...,3}{
            \draw (\x b) edge[color=myred] (\y c);
        }
    }

    \foreach \y in {1,2,3}{
        \draw (1a) edge[color=myblue] (\y c);
    }

    \draw (1b) edge[color=myblue] (2b);
    
    \draw (1c) edge[color=mygreen] (2c);
    \draw (2c) edge[color=mygreen] (3c);
    \draw (3c) edge[color=mygreen] (1c);

    \draw (1b) edge[color=myblue] (3c);

    \begin{scope}[xshift=-0.866*\c cm, yshift=0.5*\c cm]
        \node[mycircle] (1a) at (0,0){};
    \end{scope}
    
    \begin{scope}[xshift=0.866*\c cm, yshift=0.5*\c cm]
        \foreach \x/\y in {0/1b,180/2b}{
            \node[mycircle] (\y) at (canvas polar cs: radius=\r cm,angle=\x-45){};
        }
    \end{scope}
    
    \begin{scope}[yshift=-\c cm]
        \foreach \x/\y in {0/1c,120/2c,240/3c}{
            \node[mycircle] (\y) at (canvas polar cs: radius=\r cm,angle=\x+60){};
        }
    \end{scope}
    
\end{tikzpicture}}
    \hspace{2cm}
    \begin{tikzpicture}
    \def\c{2}
    \def\r{0.8}
    \begin{scope}[xshift=-0.866*\c cm, yshift=0.5*\c cm]
        \foreach \x/\y in {0/1a,120/2a,240/3a}{
            \node[mycircle] (\y) at (canvas polar cs: radius=\r cm,angle=\x){};
            \node at (-\r,\r) {$V_1$};
        }
    \end{scope}
    
    \begin{scope}[xshift=0.866*\c cm, yshift=0.5*\c cm]
        \foreach \x/\y in {0/1b,90/2b,180/3b,270/4b}{
            \node[mycircle] (\y) at (canvas polar cs: radius=\r cm,angle=\x+90){};
            \node at (\r,\r) {$V_2$};
        }
    \end{scope}
    
    \begin{scope}[yshift=-\c cm]
        \foreach \x/\y in {0/1c,72/2c,144/3c,216/4c, 288/5c}{
            \node[mycircle] (\y) at (canvas polar cs: radius=\r cm,angle=\x+90){};
            \node at (0,-1.2*\r) {$V_3$};
        }
    \end{scope}

    \foreach \x in {1,...,3}{
        \foreach \y in {1,...,4}{
            \draw (\x a) edge[color=mygreen, line width=0.75, opacity=0.6] (\y b);
        }
    }

    \foreach \x in {1,...,4}{
        \foreach \y in {1,...,5}{
            \draw (\x b) edge[color=myred,  line width=0.75, opacity=0.6] (\y c);
        }
    }

    \foreach \x in {1,...,3}{
        \foreach \y in {1,...,5}{
            \draw (\x a) edge[color=myblue, line width=0.75, opacity=0.6] (\y c);
        }
    }

    \foreach \x in {1,2,3}{
        \foreach \y in {\x,...,3}{
            \ifnum\x<\y \draw (\x a) edge[color=myred] (\y a); \fi
        }
    }
    \foreach \x in {1,...,4}{
        \foreach \y in {\x,...,4}{
            \ifnum\x<\y \draw (\x b) edge[color=myblue] (\y b); \fi
        }
    }
    \foreach \x in {1,...,5}{
        \foreach \y in {\x,...,5}{
            \ifnum\x<\y \draw (\x c) edge[color=mygreen] (\y c); \fi
        }
    }

    \draw (4b) edge[color=myblue] (4c);

    \begin{scope}[xshift=-0.866*\c cm, yshift=0.5*\c cm]
        \foreach \x/\y in {0/1a,120/2a,240/3a}{
            \node[mycircle] (\y) at (canvas polar cs: radius=\r cm,angle=\x){};
        }
    \end{scope}
    
    \begin{scope}[xshift=0.866*\c cm, yshift=0.5*\c cm]
        \foreach \x/\y in {0/1b,90/2b,180/3b,270/4b}{
            \node[mycircle] (\y) at (canvas polar cs: radius=\r cm,angle=\x+90){};
        }
    \end{scope}
    
    \begin{scope}[yshift=-\c cm]
        \foreach \x/\y in {0/1c,72/2c,144/3c,216/4c, 288/5c}{
            \node[mycircle] (\y) at (canvas polar cs: radius=\r cm,angle=\x+90){};
        }
    \end{scope}
    
\end{tikzpicture}
    \caption{Counterexamples for $k=3$ with $t=1$ (left) and $t=2$ (right). When $k=3$ and $t=1$ the construction gives a graph isomorphic to the counterexample of Pardey and Rautenbach~\cite{pardey2022}.}
    \label{fig:kn_counterexample_odd}
\end{figure}

\begin{figure}[h]
    \centering
    \raisebox{0.5cm}{\begin{tikzpicture}
    \def\c{2}
    \def\r{0.6}
    \begin{scope}[xshift=-\c cm, yshift=\c cm]
        \foreach \x/\y in {0/1a,180/2a}{
            \node[mycircle] (\y) at (canvas polar cs: radius=\r cm,angle=\x+45){};
            \node at (-0.5*\r,0.5*\r) {$V_1$};
        }
    \end{scope}
    
    \begin{scope}[xshift=\c cm, yshift=\c cm]
        \foreach \x/\y in {0/1b,180/2b}{
            \node[mycircle] (\y) at (canvas polar cs: radius=\r cm,angle=\x-45){};
            \node at (0.5*\r,0.5*\r) {$V_2$};
        }
    \end{scope}
    
    \begin{scope}[xshift=\c cm, yshift=-\c cm]
        \foreach \x/\y in {0/1c,180/2c}{
            \node[mycircle] (\y) at (canvas polar cs: radius=\r cm,angle=\x+45){};
            \node at (0.5*\r,-0.5*\r) {$V_3$};
        }
    \end{scope}

    \begin{scope}[xshift=-\c cm, yshift=-\c cm]
        \foreach \x/\y in {0/1d,180/2d}{
            \node[mycircle] (\y) at (canvas polar cs: radius=\r cm,angle=\x-45){};
            \node at (-0.5*\r,-0.5*\r) {$V_4$};
        }
    \end{scope}

    \foreach \x in {1,2}{
        \foreach \y in {1,2}{
            \draw (\x a) edge[color=mygreen, line width=0.7, opacity=0.6] (\y b);
        }
    }
    \foreach \x in {1,2}{
        \foreach \y in {1,2}{
            \draw (\x c) edge[color=mygreen, line width=0.7, opacity=0.6] (\y d);
        }
    }

    \foreach \x in {1,2}{
        \foreach \y in {1,2}{
            \draw (\x b) edge[color=myblue,  line width=0.7, opacity=0.6] (\y c);
        }
    }
    \foreach \x in {1,2}{
        \foreach \y in {1,2}{
            \draw (\x a) edge[color=myblue,  line width=0.7, opacity=0.6] (\y d);
        }
    }

    \foreach \x in {1,2}{
        \foreach \y in {1,2}{
            \draw (\x a) edge[color=violet, line width=0.7, opacity=0.6] (\y c); 
        }
    }

    \foreach \x in {1,2}{
        \foreach \y in {1,2}{
            \draw (\x b) edge[color=myred, line width=0.75, opacity=0.6] (\y d);
        }
    }

    \foreach \x in {1,2}{
        \foreach \y in {\x,2}{
            \ifnum\x<\y \draw (\x a) edge[color=myred] (\y a); \fi
            \ifnum\x<\y \draw (\x c) edge[color=myred] (\y c); \fi
        }
    }
    \foreach \x in {1,2}{
        \foreach \y in {\x,2}{
            \ifnum\x<\y \draw (\x b) edge[color=violet!75!white] (\y b); \fi
        }
    }
    \foreach \x in {1,2}{
        \foreach \y in {\x,2}{
            \ifnum\x<\y \draw (\x d) edge[color=violet!75!white] (\y d); \fi
        }
    }

    \draw (2d) edge[color=myred] (2a);

    \draw (2d) edge[color=violet!75!white] (1c);

    \begin{scope}[xshift=-\c cm, yshift=\c cm]
        \foreach \x/\y in {0/1a,180/2a}{
            \node[mycircle] (\y) at (canvas polar cs: radius=\r cm,angle=\x+45){};
        }
    \end{scope}
    
    \begin{scope}[xshift=\c cm, yshift=\c cm]
        \foreach \x/\y in {0/1b,180/2b}{
            \node[mycircle] (\y) at (canvas polar cs: radius=\r cm,angle=\x-45){};
        }
    \end{scope}
    
    \begin{scope}[xshift=\c cm, yshift=-\c cm]
        \foreach \x/\y in {0/1c,180/2c}{
            \node[mycircle] (\y) at (canvas polar cs: radius=\r cm,angle=\x+45){};
        }
    \end{scope}

    \begin{scope}[xshift=-\c cm, yshift=-\c cm]
        \foreach \x/\y in {0/1d,180/2d}{
            \node[mycircle] (\y) at (canvas polar cs: radius=\r cm,angle=\x-45){};
        }
    \end{scope}
\end{tikzpicture}}
    \hspace{2cm}
    \begin{tikzpicture}
    \def\c{2}
    \def\r{0.8}
    \begin{scope}[xshift=-\c cm, yshift=\c cm]
        \foreach \x/\y in {0/1a,90/2a,180/3a,270/4a}{
            \node[mycircle] (\y) at (canvas polar cs: radius=\r cm,angle=\x){};
            \node at (-\r,\r) {$V_1$};
        }
    \end{scope}
    
    \begin{scope}[xshift=\c cm, yshift=\c cm]
        \foreach \x/\y in {0/1b,72/2b,144/3b,216/4b, 288/5b}{
            \node[mycircle] (\y) at (canvas polar cs: radius=\r cm,angle=\x+90){};
            \node at (\r,\r) {$V_2$};
        }
    \end{scope}
    
    \begin{scope}[xshift=\c cm, yshift=-\c cm]
        \foreach \x/\y in {0/1c,90/2c,180/3c,270/4c}{
            \node[mycircle] (\y) at (canvas polar cs: radius=\r cm,angle=\x){};
            \node at (\r,-\r) {$V_3$};
        }
    \end{scope}

    \begin{scope}[xshift=-\c cm, yshift=-\c cm]
        \foreach \x/\y in {0/1d,120/2d,240/3d}{
            \node[mycircle] (\y) at (canvas polar cs: radius=\r cm,angle=\x+90){};
            \node at (-\r,-\r) {$V_4$};
        }
    \end{scope}

    \foreach \x in {1,...,4}{
        \foreach \y in {1,...,5}{
            \draw (\x a) edge[color=mygreen, line width=0.7, opacity=0.35] (\y b);
        }
    }
    \foreach \x in {1,...,4}{
        \foreach \y in {1,...,3}{
            \draw (\x c) edge[color=mygreen, line width=0.7, opacity=0.35] (\y d);
        }
    }

    \foreach \x in {1,...,5}{
        \foreach \y in {1,...,4}{
            \draw (\x b) edge[color=myblue,  line width=0.7, opacity=0.35] (\y c);
        }
    }
    \foreach \x in {1,...,4}{
        \foreach \y in {1,...,3}{
            \draw (\x a) edge[color=myblue,  line width=0.7, opacity=0.35] (\y d);
        }
    }

    \foreach \x in {1,...,4}{
        \foreach \y in {1,...,4}{
        \ifnum\x=\y
        \else
            \draw (\x a) edge[color=violet, line width=0.7, opacity=0.35] (\y c);
        \fi   
        }
    }

    \foreach \x in {1,...,5}{
        \foreach \y in {1,...,3}{
            \draw (\x b) edge[color=myred, line width=0.75, opacity=0.35] (\y d);
        }
    }

    \foreach \x in {1,...,4}{
        \foreach \y in {\x,...,4}{
            \ifnum\x<\y \draw (\x a) edge[color=myred] (\y a); \fi
            \ifnum\x<\y \draw (\x c) edge[color=myred] (\y c); \fi
        }
    }
    \foreach \x in {1,...,5}{
        \foreach \y in {\x,...,5}{
            \ifnum\x<\y \draw (\x b) edge[color=violet!75!white] (\y b); \fi
        }
    }
    \foreach \x in {1,...,3}{
        \foreach \y in {\x,...,3}{
            \ifnum\x<\y \draw (\x d) edge[color=violet!75!white] (\y d); \fi
        }
    }

    \draw (1d) edge[color=myred] (3a);
    \draw (1d) edge[color=myred] (4a);

    \draw (1d) edge[color=violet!75!white] (3c);
    \draw (1d) edge[color=myred] (4c);

    \begin{scope}[xshift=-\c cm, yshift=\c cm]
        \foreach \x/\y in {0/1a,90/2a,180/3a,270/4a}{
            \node[mycircle] (\y) at (canvas polar cs: radius=\r cm,angle=\x){};
        }
    \end{scope}
    
    \begin{scope}[xshift=\c cm, yshift=\c cm]
        \foreach \x/\y in {0/1b,72/2b,144/3b,216/4b, 288/5b}{
            \node[mycircle] (\y) at (canvas polar cs: radius=\r cm,angle=\x+90){};
        }
    \end{scope}
    
   \begin{scope}[xshift=\c cm, yshift=-\c cm]
        \foreach \x/\y in {0/1c,90/2c,180/3c,270/4c}{
            \node[mycircle] (\y) at (canvas polar cs: radius=\r cm,angle=\x){};
        }
    \end{scope}

    \begin{scope}[xshift=-\c cm, yshift=-\c cm]
        \foreach \x/\y in {0/1d,120/2d,240/3d}{
            \node[mycircle] (\y) at (canvas polar cs: radius=\r cm,angle=\x+90){};
        }
    \end{scope}
\end{tikzpicture}
    \caption{Counterexamples for $k=4$ with $t=1$ (left) and $t=2$ (right).}
    \label{fig:kn_counterexample_even}
\end{figure}

\section{Open Problems}\label{sec:open_problems}
There are several interesting directions for further research. There is a gap between the $\bigO(k^2)$ bound obtained for representative matchings in bipartite graphs in \cref{thm:vec-bipartite}, and the worst-case lower bound of $\sqrt{k/2}$ constructed for colour-balanced matchings in \cref{thm:lower-Knn}. Our use of \cref{thm:vec-bipartite} to prove each of our other results for representative subgraphs implies that any improvement to the general $k^2$ bound in \cref{thm:vec-bipartite} immediately gives an equal improvement to the $k^2$ factor in the bounds of each of our other results. We would therefore be very interested to know the true bound. In the specific case of colour-balanced matchings,  we conjecture that our lower bound is optimal.
\begin{conjecture}
    For all positive integers $k$ and $t$, every colour-balanced $k$-edge-coloured $K_{kt,kt}$ admits a perfect matching $M$ satisfying $$f_c(M) = \bigO(\sqrt{k}).$$
\end{conjecture}
There is also a gap between the $\bigO(k^2)$ upper bound obtained in \cref{thm:complete_colours} and the constant lower bound constructed in \cref{thm:lower_Kn} for matchings in complete graphs. We again conjecture that the lower bound is correct. If this is true, then our $\sqrt{k}$ lower bound for perfect matchings of $K_{kt,kt}$ implies that new ideas will be necessary to prove it.
\begin{conjecture}
    There exists some absolute constant $C$ such that for all positive integers $k$ and $t$, every colour-balanced $k$-edge-colouring of $K_{2kt}$ admits a perfect matching $M$ satisfying $f_c(M) < C$.
\end{conjecture}
Improvements via other means to our remaining error bounds would also be of interest. We highlight two cases in particular. First, when $k=2$, error-bounds that are sharp up to an additive constant are known for embeddings of fixed spanning forests of maximum degree $\Delta$ (see \cite{hollom2024}). While we have shown that for $k > 2$ the error remains linear in $\Delta$, we have made no effort to optimise the coefficient.
\begin{problem}
    Improve the bounds on $f_\vecf(F)$ for bounded-degree spanning forests $F$ in complete graphs when $k \geq 3$.
\end{problem}
Second, since the maximum degree of a Hamilton cycle is $2$, our general upper bound in \cref{thm:bounded_deg_bound} implies that the error for embedding a representative Hamilton cycle in a complete graph is $\bigO(k^2)$. Our methods imply that this bound also holds for Hamilton cycles in complete bipartite graphs. Suppose we know an almost representative Hamilton cycle exists with error $\varepsilon$. Using a similar splitting strategy to the one applied in \cref{sec:knn_section}, we can apply Alon's necklace theorem to obtain an almost representative perfect matching with error $\varepsilon + \bigO(k)$. It follows that any improvement to the $k^2$ upper bound for Hamilton cycles yields an improvement to the corresponding upper bound for perfect matchings.
\begin{problem}
    Improve the bounds on $f_\vecf(H)$ for a Hamilton cycle $H$ in complete or complete bipartite graphs.
\end{problem}

Finally, we note that the matching problem in the bipartite setting is of independent interest. The question relates closely to recent work on finding rainbow matchings in colour-balanced $n$-edge-coloured complete bipartite graphs. When $K_{n,n}$ is $n$-edge-coloured by some colour-balanced $c$, a colour-balanced perfect matching of $K_{n,n}$ is precisely a rainbow perfect matching. When $c$ is a proper $n$-edge-colouring, this is precisely the problem of finding a Latin square transversal of size $n$, and when $c$ is a non-proper $n$-edge-colouring, this is the problem of finding an equi-$n$-square transversal of size $n$. Asymptotically sharp bounds for both problems are known (see Montgomery~\cite{montgomery2023proof} and Chakraborti, Christoph, Hunter, Montgomery and Petrov~\cite{chakraborti2024} for the most up-to-date results). There are many other interesting questions in these settings that extend naturally to our $k$-edge-coloured context. 

\subsection*{Acknowledgements}
The authors thank Leslie Goldberg, Lukas Michel and Youri Tamitegama for many helpful conversations. Thanks also to V\'{a}clav Rozho\v{n} for pointing out an error in an earlier version of this paper.
\printbibliography
\end{document}